\documentclass{amsart}   	

\usepackage{amsmath}
\usepackage{amssymb}
\usepackage{amsfonts}
\usepackage{hyperref}
\usepackage{tikz-cd}
\usepackage{amscd}
\usepackage[all]{xy}
\usepackage{enumerate}

\def\Z{{\mathbb Z}} 
\def\Q{{\mathbb Q}} 
 
\def\C{{\mathbb C}} 
 
\def\P{{\mathbb P}} 
\def\F{{\mathbb F}} 
\def\SS{{\mathbb S}} 
\def\V{{\mathbb V}}

\def\cC{{\mathcal C}}

\def\cD{{\mathcal D}}
\def\d{{\mathfrak{d}}}
\def\E{{\mathbb E}}
\def\cE{{\mathcal{E}}}

\def\cG{{\mathcal{G}}}

\def\hyp{{\mathrm{hyp}}}
\def\H{{\mathcal H}}
\def\I{{\mathcal I}}
\def\J{{\mathcal J}}
\def\L{{\mathbb{L}}}
\def\M{{\mathcal M}}
\def\MHS{{\mathsf{MHS}}}
\def\cP{{\mathcal P}}
\def\p{{\mathfrak{p}}}
\def\prim{{\mathrm{prim}}}
\def\reg{{\mathrm{reg}}}

\def\T{{\mathcal T}}
\def\U{{\mathcal U}}
\def\u{{\mathfrak{u}}}
\def\V{{\mathbb{V}}}
\def\cV{{\mathcal V}}
\def\X{{\mathcal X}}
\def\Z{{\mathbb{Z}}}
\def\cZ{{\mathcal Z}}

\def\D{{\Delta}} 
\def\G{{\Gamma}}

\def\dot{{\bullet}} 
\def\bs{\backslash} 

\def\ab{\mathrm{ab}}
\def\adj{\mathrm{adj}}
\def\hyp{\mathrm{hyp}}

\def\orb{\mathrm{orb}}

\def\un{\mathrm{un}}

\newcommand{\Ext}{\operatorname{Ext}}
\newcommand{\Gr}{\operatorname{Gr}}
\newcommand{\Hom}{\operatorname{Hom}}

\newcommand{\Der}{\operatorname{Der}}

\newcommand\id{\operatorname{id}}

\newcommand\Aut{\operatorname{Aut}} 
 
\newcommand\Jac{\operatorname{Jac}}

\newcommand\Sp{\operatorname{Sp}} 

\newtheorem{theorem}{Theorem}[section]
\newtheorem{lemma}[theorem]{Lemma}
\newtheorem{proposition}[theorem]{Proposition}
\newtheorem{corollary}[theorem]{Corollary}
\newtheorem{bigtheorem}{Theorem}

\theoremstyle{definition}
\newtheorem{definition}[theorem]{Definition}
\newtheorem{example}[theorem]{Example}

\theoremstyle{remark}
\newtheorem{remark}[theorem]{Remark}

\begin{document}

\title{Remarks on Collino cycles
and hyperelliptic Johnson homomorphisms
}
\author{
Ma Luo \and Tatsunari Watanabe
}
\address{School of Mathematical Sciences, East China Normal University, Shanghai}
\email{mluo@math.ecnu.edu.cn}
\address{Mathematics Department, Embry-Riddle Aeronautical University, Prescott}
\email{watanabt@erau.edu}

\thanks{The first author is supported partly by Science and Technology Commission of Shanghai Municipality (No. 22DZ2229014), and partly by National Natural Science Foundation of China, Grant No. 12201217.}

\begin{abstract}
\smallskip
A Collino cycle is a higher cycle on the Jacobian of a hyperelliptic curve. The universal family of Collino cycles naturally gives rise to a normal function, whose induced monodromy relates to the hyperelliptic Johnson homomorphism. Colombo computed this monodromy explicitly and made this relation precise. We recast this in the perspective of relative completion. In particular, we use Colombo's result to construct Collino classes, which are cohomology classes of hyperelliptic mapping class groups with coefficients in a certain symplectic representation. We also determine the dimension of their span in the case of the level two hyperelliptic mapping class group.
\end{abstract}
\maketitle

\setcounter{tocdepth}{1}
\tableofcontents

\section{Introduction}
Denote the mapping class group of a compact topological surface $S$ of genus $g$ with $n$ distinct marked points by $\G_{g,n}$. For $g \geq 1$, there is a surjective natural representation $\G_{g,n}\to \Sp(H_1(S, \Z))$ and its kernel is called the Torelli group, denoted by $T_{g,n}$. 

For a compact Riemann surface $C$ (which we call a curve in this paper), let $\Jac C$ be its jacobian. Then using a point $x$ in $C$, we have the algebraic cycle $C_x-C^{-}_x$ in $\Jac C$ called the Ceresa cycle.  Let $\T_{g,n}$ be the Torelli space of marked, $n$-pointed curves of genus $g$. There is a bundle $\J_1\to \T_{g,n}$ over $\T_{g,n}$ of the intermediate jacobians whose fiber over $[C]$ is $J_1(\Jac C)$. The cycle $C_x-C^{-}_x$ determines a point $e_{C,x}$ in the intermediate jacobian $J_1(\Jac C)$. This construction extends to families and defines a section \textcolor{black}{$e_{g,1}:\T_{g,1}\to \J_1$} of the bundle. This section is by definition a normal function (see Definition \ref{nf}). It induces a homomorphism of fundamental groups \textcolor{black}{$\xi_{g,1}: T_{g,1}\to H_3(\Jac C, \Z) = \Lambda^3H_1(C)$}, which is equal to twice the Johnson homomorphism (see \cite[\S 6]{hain}).  On the other hand, the Johnson homomorphism yields a nontrivial class in $H^1(\G_{g,1}, H_1(S, \Q))$ and $H^1(\G_{g,1}, \Lambda^3H_1(S,\Q)/\theta\wedge H_1(S,\Q))$ (see \cite{mor}, \cite[\S 5]{hm}). 

When $C$ is hyperelliptic and $x$ is a Weierstrass point, its corresponding Ceresa cycle $C_x-C^{-}_x$ is trivial, and in general its image in the primitive jacobian $J_1(\Jac C)_\prim$ is trivial. So instead, we will consider a canonical higher cycle $(Z, q_1, q_2)$ associated to $C$ with two ordered distinct Weierstrass points $q_1$ and $q_2$ constructed in \cite{collino} by Collino. The higher cycle $Z$ can be viewed as a degeneration of the Ceresa cycle for the stable curve obtained from $C$ by gluing $q_1$ and $q_2$. Collino proves in \cite{collino} that the regulator image, $\mathrm{reg}(Z)$, of $Z$ is nontrivial for general hyperelliptic curves. In \cite{colombo}, Colombo constructs an extension class \textcolor{black}{$Pe$} associated to $C$ with $q_1$ and $q_2$, which is equal to $(2g+1)$$\mathrm{reg}(Z)$ in the primitive intermediate jacobian $I_2(\Jac C)_{\prim}$. Denote the hyperelliptic Torelli space by $\H_g[0]$. There is a normal function \textcolor{black}{$P\cE:\H_g[0] \to \I_{2\prim}$} extending the class \textcolor{black}{$Pe$} and Colombo computes its monodromy action using higher Johnson homomorphisms, which we call hyperelliptic Johnson homomorphisms in this paper. 

For a Weierstrass point $q$ of a hyperelliptic curve $C$, denote the Lie algebra of the unipotent completion of $\pi_1(C, q)$ over $\Q$ by $\p$ and its derivation algebra by $\Der\p$. The Lie algebras $\p$ and $\Der\p$ admit weight filtrations $W_\bullet\p$ and $W_\bullet\Der\p$ from Hodge theory (see \cite{hain3}). 
Fix a hyperelliptic involution $\sigma$ of $S$. The hyperelliptic mapping class group $\Delta_g$ is defined as the subgroup of $\G_g$ consisting of elements that commute with the class $[\sigma]$. The hyperelliptic Torelli group denoted by $T\Delta_g$ is given by the intersection $\Delta_g\cap T_g$ in $\G_g$. As a higher Johnson homomorphism, we have the hyperelliptic Johnson homomorphism $T\Delta_g\to \Gr^W_{-2}\Der\p$, denoted by $\tau^\hyp_q$. Composing $\tau^\hyp_q$ with a certain $\Sp(H_1(C, \Q))$-equivariant projection of $\Gr^W_{-2}\Der\p$ onto $\Lambda^2H_1(C,\Q)/\langle\theta\rangle$, we obtain an $\Sp(H_1(C, \Q))$-equivariant homomorphism, which we denote by $\tilde\tau^\hyp_q$. Denote the normal function extending $\mathrm{reg}(Z)$ by \textcolor{black}{$\widetilde{R}_{\cZ}$} and the projection onto the fiber by $p_{I_{2\prim}}$. Their composition $p_{I_{2\prim}}\circ \widetilde{R}_\cZ$ induces a homomorphism of fundamental groups $T\Delta_g\to \Lambda^2H_1(C,\Z)/\langle\theta\rangle$ by $\pi_{\cZ}$. As a remark on Colombo's work, we prove

\begin{bigtheorem}\label{hyp johnson and normal function for Z} With notation as above,  if $g \geq 2$, then 
$$\tilde\tau^\hyp_{q_2} - \tilde\tau^\hyp_{q_1} = (g+1)\pi_{\cZ}.$$

\end{bigtheorem}
The homomorphism $\tilde\tau^\hyp_q$ is $\Sp(H_1(C, \Q))$-equivariant  
and yields a nontrivial class, denoted by $[q]$, in $H^1(\Delta_g[2], \Lambda^2H_1(C,\Q)/\langle\theta\rangle)$, where $\Delta_g[2]$ is the level 2 hyperelliptic mapping class group. We call the class $[q]$ a Weierstrass class. Due to Theorem \ref{hyp johnson and normal function for Z}, the homomorphism $(g+1)\pi_{\cZ}$ produces a nontrivial class given by $[q_2] - [q_1]$, which we call a Collino class in $H^1(\Delta_g[2], \Lambda^2H_1(C,\Q)/\langle\theta\rangle)$. The fact that $\mathrm{reg}(Z)$ is nontrivial for general hyperelliptic curves is equivalent to that the corresponding Collino class is nontrivial. Our second main result is concerned with the subspace, denoted by $X_\zeta$, of $H^1(\Delta_g[2], \Lambda^2H_1(C,\Q)/\langle\theta\rangle)$ spanned by all Collino classes. Denote the subspace of $H^1(\Delta_g[2], \Lambda^2H_1(C,\Q)/\langle\theta\rangle)$ spanned by all Weierstrass classes by $X_\omega$. 
\begin{bigtheorem} With notation as above, if $g\geq 2$, then $X_\zeta =X_\omega$ and $\mathrm{dim}~X_\zeta = 2g+1$. 
\end{bigtheorem}
In fact, the $2g+2$ Weierstrass classes satisfy a single equation 
$\sum_{i=1}^{2g+2}[q_i] =0$, and each class $(2g+2)[q_i]$ is an integral combination of $(2g+1)$ Collino classes (see Remark \ref{main remark}). From the Teichm\"uller theory, it is known that, for $g =2$ by Hubbard in \cite{hub} and for $g > 2$ by Earle and Kra in \cite{EaKr}, the universal curve over a branch of the hyperelliptic locus in the Teichm\"uller space  admits exactly $2g+2$ Weierstrass sections.  Our result implies that  Weierstrass sections satisfy an algebraic relation via the hyperelliptic Johnson homomorphisms. On the other hand, the tautological sections of the universal curve yield linearly independent classes in $H^1(\G_{g,n}, H_1(S, \Q))$ (cf. \cite[Prop.~12.1]{hai5}). It is of our interest to investigate further how the relative completion of the level 2 hyperelliptic mapping class group $\Delta_g[2]$ determines the Weierstrass sections (cf. \cite{wat}). 
\bigskip

{\it Acknowledgments:} We are grateful to Richard Hain for helpful discussions on the relative completions of hyperelliptc mapping class groups and letting us use a part of an unpublished notes on the hyperelliptic mapping class groups \cite{hkw}. 


\section{Mapping class groups}\label{mapping class group}
Fix a smooth, compact oriented surface $S$ of genus $g$ and a subset $P$ of $S$ consisting of  $n$ distinct points of $S$.
Assume that $2g-2+n>0$. The mapping class group $\G_{S, P}$ is 
defined to be the group of isotopy classes of 
orientation-preserving diffeomorphisms of $S$ that fix $P$ pointwise. 
By the classification of surfaces, $\G_{S,P}$ is independent of a choice of the pair $(S, P)$, and hence we denote $\G_{S, P}$ by $\G_{g,n}$. 
When $n=0$, we denote $\G_{g,0}$ by $\G_g$. 
\subsection{Level structures} Denote $H_1(S;R)$ by $H_R$ where $R$ is a commutative ring. Denote the algebraic intersection 
pairing on $H_R$ by $\langle~,~\rangle:H_R^{\otimes 2} \to R$. It is a unimodular symplectic form. Set 
$$ 
\Sp(H_R) = \Aut(H_R,\langle~,~\rangle). 
$$ 

Fixing a symplectic basis $a_1,\dots,a_g,b_1,\dots,b_g$ of $H_R$ gives 
an isomorphism $\Sp(H_R)$ with the classical symplectic group $\Sp_g(R)$ consisting of
$2g\times 2g$ symplectic matrices. 
The principal congruence subgroup $\Sp(H_\Z)[m]$ of $\Sp(H_\Z)$ of level $m\in 
\Z$ is the kernel of the reduction mod $m$ mapping: 
$$ 
\Sp(H_\Z)[m] := \ker\{\Sp(H_\Z) \to \Sp(H_{\Z/m\Z})\}. 
$$ 
Fix a point $q$ in $S$. The action of $\G_{g,n}$ on $\pi_1(S, q)$ induces an action on $H_\Z$ that preserves the intersection pairing. Therefore, there is a representation 
$$ 
\rho : \G_{g,n} \to \Sp(H_\Z). 
$$ 
This is well known to be surjective \textcolor{black}{when $R=\Z$} (e.g. \cite[Thm.~6.4]{FaMa}). 
For each integer $m\ge 0$, we define the level $m$ subgroup of $\G_{g,n}$ to be 
the kernel of the reduction of $\rho$ mod $m$: 
$$ 
\G_{g,n}[m] = \ker\{\G_{g,n} \to \Sp(H_{\Z/m\Z})\}. 
$$ 
When $m=1$, we omit the level notation, so $\G_{g,n}[1] = \G_{g,n}$. \\
The Torelli group $T_{g,n}$ is defined to be the kernel of $\rho$ and it is the 
the level $0$ subgroup of $\G_{g,n}$: 
$$ 
T_{g,n} = \G_{g,n}[0] = \ker\{\G_{g,n} \to \Sp(H_{\Z})\}. 
$$ 
The Torelli groups are torsion free (e.g. \cite[Thm.~6.12]{FaMa}).  Since 
$\Sp(H_\Z)[m]$ is torsion free for all $m\ge 3$, it follows that $\G_{g,n}[m]$ 
is torsion free for all $m\ge 3$.

\section{Hyperelliptic mapping class groups}
The content of this section comes from an unpublished notes on the completions of hyperelliptic mapping class groups \cite{hkw}. 

Fix a hyperelliptic involution $\sigma:S \to S$.
It is an orientation-preserving diffeomorphism of order 2 of $S$ with exactly $2g + 2$ fixed points, which we call Weierstrass points. 
 The quotient space $S/\langle \sigma \rangle$ is a sphere by the Riemann-Hurwicz formula. Therefore, it then follows that all hyperelliptic involutions are conjugate in the group of orientation preserving diffeomorphisms of $S$, denoted by $\mathrm{Diff}^+S$. \\
\indent An orientation-preserving diffeomorphism of $S$ is said to be {\it symmetric} if it commutes with $\sigma$. The hyperelliptic mapping class group, denoted by $\Delta_g$, is defined to be the group of isotopy classes of orientation-preserving symmetric diffeomorphisms of $S$:
$$\Delta_g :=\pi_0(\text{centralizer of }\sigma \text{ in }\mathrm{Diff}^+S)$$
The following result by Birman and Hilden allows us to consier $\Delta_g$ as a subgroup of $\G_g$. 
\begin{theorem}[Birman-Hilden \cite{birman-hilden}] 
The natural homomorphism $\D_g \to \G_g$ is injective. Its image is the 
centralizer of the isotopy class of $\sigma$ in $\G_g$. 
\end{theorem} 

\subsection{Level 2 hyperelliptic mapping class group}\label{level 2 hyp map class fixing w points}
  Denote the set of fixed points of  $\sigma$ by $W$. The action of $\Delta$ on $W$ yields a homomorphism $\rho_W : \Delta_g \to  \mathrm{Aut}W$. It follows from \cite[Thm.~1]{birman-hilden} that the homomorphism $\rho_W$ is surjective and that there is 
an extension:
\begin{equation}
\label{extn}
1 \to \langle \sigma \rangle \to \ker \rho_W \to \G_{0,2g+2} \to 1.
\end{equation}
To identify the kernel, we need to recall the connection between labeling of
the Weierstrass points and level 2 structures on hyperelliptic curves. (Cf.\
\cite[p.~3.31]{mumford}.)

Denote the $\F_2$-valued characteristic function $W \to \F_2$ of a subset $T$
of the set of Weierstrass points by $e_T$. There is an $\F_2$ linear mapping
$$
q : \{e_T : T \subseteq W,\ |T| \text{ is even}\} \to H_1(S,\F_2)=H_{\F_2}.
$$
It is defined as follows: let $f : S \to S^2$ be the 2-to-1 mapping whose
Galois group is generated by $\sigma$. This induces an isomorphism of $W$ with
the set of critical values of $f$. If $T$ is an even subset of $W$ then there
is a 1-chain $\gamma$ in $S^2$ whose mod 2 boundary is $f(T)$. Any two such
chains differ by a boundary mod 2. Define $q(e_T)$ to be the class in
$H_1(S,\F_2)$ of $f^{-1}(\gamma)$. Its homology class is well defined mod 2.

\begin{lemma}
The mapping $q$ induces a linear isomorphism
$$
\overline{q} :
\{e_T : T\subseteq W,\ |T| \equiv 0 \bmod 2\}/
\{e_T - e_{T^c}: |T| \equiv 0 \bmod 2\}
\to H_{\F_2}
$$
where $T^c$ denotes the complement $W-T$ of $T$. \qed
\end{lemma}

The intersection pairing on $H_{\F_2}$ corresponds to the pairing
$$
e_S \cdot e_T = \# (S\cap T) \bmod 2
$$
on the even subsets of $W$. Thus every automorphism of $W$ induces a symplectic
automorphism of $H_{\F_2}$.

\begin{corollary}
There is a natural faithful representation $\Aut W \hookrightarrow
\Sp(H_{\F_2})$.
\end{corollary}
Denote the restriction of $\G_g$ to $\Delta_g$ by 
$$
\rho^\hyp:\Delta_g\to \Sp(H_\Z).
$$
For each $m \geq 0$, we define the level $m$ subgroup $\Delta_g[m]$ of $\Delta_g$ by the reduction of $\rho^\hyp$ mod $m$:
$$
\Delta_g[m] = \ker(\Delta_g\to \Sp(H_{\Z/m\Z})).
$$
\begin{remark} 
Since every curve of genus 2 is hyperelliptic, $\G_2[m] \cong \D_2[m]$ for all 
$m\ge 0$. 
\end{remark} 
\begin{corollary}
The image of the natural homomorphism $\D_g \to \Sp(H_{\F_2})$ is $\Aut W$.
Consequently, the kernel of $\rho_W : \D_g \to \Aut W$ is $\D_g[2]$. \qed
\end{corollary}

One therefore has extensions
\begin{align}
\label{extensions}
\begin{CD}
1 @>>> \langle \sigma \rangle @>>> \D_g[2] @>>> \G_{0,2g+2} @>>> 1
\cr
1 @>>> \D_g[2] @>>> \D_g @>>> \Aut W @>>> 1
\end{CD}
\end{align}

Since $\G_g[m]$ is torsion free for all $m\ge 3$, the same is true for 
$\D_g[m]$. The group $\D_g[0]$ is the kernel of $\rho^\hyp$. We shall call it the {\it hyperelliptic Torelli group} and denote 
it by $T\D_g$.

\subsection{Image of $\rho^\hyp$}
Consider $\Aut W$ as a subgroup of $\Sp(H_{\Z/2\Z})$ via the above faithful representation.  Denote the inverse image of $\Aut W$ under the reduction map $\Sp(H_\Z)\to \Sp(H_{\Z/2\Z})$ by $G_g$. 
We have the following diagram.
$$ 
\begin{CD} 
G_g @>>> \Sp(H_\Z) \cr 
@VVV @VVV \cr 
\Aut W @>>> \Sp(H_{\Z/2\Z}) 
\end{CD} 
$$
Here we recall the following result of A'Campo (see also \cite[\S8, Lemma 8.12]{mumford}).

\begin{theorem}[{\cite{acampo}}] \label{image of hyp mapping class group}
If $g\ge 2$, then the image of $\rho^\hyp : \D_g \to \Sp(H_\Z)$ is $G_g$. In 
particular, the image of the restriction of $\rho^\hyp$ to $\D_g[2]$ is $\Sp(H_\Z)[2]$. 
\end{theorem} 




\subsection{Subgroups $\Delta_{g,n}$} 
Fixing a labeling $W \cong \{1,2,\dots,2g+2\}$ determines an isomorphism of $\Aut W$ with the symmetric group 
$\SS_{2g+2}$. 
Fix a Weierstrass point $q \in W$. Denote its stabilizer in $\Aut W$ by $\Aut_q W$.
This is isomorphic to the symmetric group $\SS_{2g+1}$. The group $\D_{g,1}$ is
defined by
$$
\D_{g,1} := (\rho_W)^{-1}(\Aut_q W).
$$
It contains $\D_g[2]$ and is an extension
\begin{equation}
\label{red-ext}
1 \to \D_g[2] \to \D_{g,1} \to \Aut_q W \to 1.
\end{equation}
More generally, let $Q$ be a subset of $W$ with $|Q| =n$. Denote the subgroup of $\Aut W$ fixing $Q$ pointwise by $\Aut_QW$. This group is isomorphic to the symmetric group $\SS_{2g+2-n}$. By abuse of notation, for any such subset $Q$ of $W$, the subgroup $\D_{g,n}$ of $\Delta_g$ is 
defined to be the inverse image of $\Aut_QW$ under $\rho_W$: 
$$ 
\D_{g,n} = (\rho_W)^{-1}(\Aut_Q W). 
$$ 
Similarly, there is an extension 
\begin{equation*} 
\label{red-ext for Q} 
1 \to \D_g[2] \to \D_{g,n} \to \Aut_Q W \to 1. 
\end{equation*} 
Similarly, denote the image of $\Delta_{g,n}$ under $\rho^\hyp$ in $\Sp(H_\Z)$ by $G_{g,n}$. 
Then, for $0\leq n \leq 2g+2$, there is a commutative diagram of extensions:
$$\xymatrix@R=1em@C=2em{
1\ar[r]&T\Delta_g\ar[r]\ar@{=}[d]&\Delta_{g,n}\ar[r]\ar[d]&G_{g,n}\ar[r]\ar[d]&1\\
1\ar[r]&T\Delta_g\ar[r]&\Delta_g\ar[r]&G_g\ar[r]&1.
}
$$
When $n = 2g+2$, we denote $\Delta_{g, 2g+2}$ by $\Delta_g[2]$. 
\subsection{Group cohomology of $\Delta_g$ and $\Delta_g[2]$}
Here we state basic results about cohomology of hyperelliptic mapping class groups $\Delta_g$ and $\Delta_g[2]$.
\begin{proposition}\label{hyp cohomology}
If  $V$ is a rational $\Delta_g$-module, then $H^\dot(\D_g[2],V)$
is an $\Aut W$-module and there are natural isomorphisms
$$
H^\dot(\D_g,V) \cong H^\dot(\D_g[2],V)^{\Aut W} 
\text{ and }H^\dot(\D_{g,1},V) \cong H^\dot(\D_g[2],V)^{\Aut_qW}.
$$
If $\sigma$ acts as $-\id$ on $V$, then $H^\dot(\D_g,V) = H^\dot(\D_g[2],V) = 0$.
If $\sigma$ acts trivially on $V$, then $V$ is a $\Gamma_{0,2g+2}$-module and
there is a natural isomorphism
$$
H^\dot(\D_g[2],V) \cong H^\dot(\G_{0,2g+2},V).
$$
\end{proposition}

\begin{proof}
The first assertions follow from the Hochschild-Serre spectral sequence of the
extensions (\ref{extensions}) and (\ref{red-ext}).
The second assertion follows from the fact that the centralizer of a group acts trivially on the cohomology.   Since $\sigma$
is central in $\D_g$, it acts as a $\D_g$-linear automorphism of $V$ as $-\id$, and
therefore of $H^\dot(\D_g,V)$ and $H^\dot(\D_g[2],V)$.  But since the centralizer acts trivially on
the cohomology of the group, it follows that $\sigma$ also acts trivially.  Thus multiplication by 2 annihilates both of these cohomology groups. Since $V$ is a rational representation, our claim follows. The third assertion follows from the
Hochschild-Serre spectral sequence of the extension (\ref{extn}).
\end{proof}





\section{Moduli spaces of hyperelliptic curves}
By a complex curve, or a curve for short, we shall mean a Riemann surface. Suppose that $2g-2+n>0$. Let $S$ be a reference surface of genus $g$ as in \S\ref{mapping class group} and $P$ a subset of $S$ consisting of $n$ points.  Denote the
Teichm\"uller space of marked, $n$-pointed, compact genus $g$ curves by
$X_{g,n}$. As a set $X_{g,n}$ is
\begin{center}
$\left\{
\parbox{2.5in}{orientation-preserving diffeomorphisms $f: S \to C$
to a complex curve}
\right\}$
\bigg/
\text{isotopies, constant on $P$}.
\end{center}
This is a contractible complex manifold of dimension $3g-3+n$. 
When $n =0$, $X_{g,0}$ is denoted by $X_g$.  
The mapping class group $\G_{g,n}$ acts on $X_{g,n}$ as a group of
biholomorphisms via its action on the markings:
$$
\phi : f \mapsto f\circ \phi^{-1},\quad \phi \in \G_{g,n},\ f\in X_{g,n}.
$$
This action is properly discontinuous and virtually free. The
isotropy group of $[(S,P) \to (C;x_1,\dots,x_n)]\in X_{g,n}$ is isomorphic to
the set of automorphisms of $C$ that fix $\{x_1,\dots,x_n\}$ pointwise.

For each $m\geq 0$,  the moduli space of $n$-pointed smooth projective curves of genus $g$ with a level $m$ structure is the quotient of
$X_{g,n}$ by the level $m$ subgroup of $\G_{g,n}$:
$$
\M_{g,n}[m] = \big(\G_{g,n}[m]\bs X_{g,n}\big)^\orb.
$$
It will be regarded as a complex analytic orbifold. It is a model of the
classifying space of $\G_{g,n}[m]$. When the group $\G_{g,n}[m]$ is torsion
free, $\M_{g,n}[m]$ is a smooth variety and also a fine moduli space for
$n$-pointed smooth projective curves of genus $g$ with a level $m$ structure.
This occurs, for example, when $m\ge 3$, $m=0$, or when $n>2g+2$. Note that $\M_{g,n}[1] = \M_{g,n}$ and that $\M_{g,n}[0]$ is the quotient of
Teichm\"uller space by the Torelli group $T_{g,n}$. It is known as {\it Torelli
space}. We shall denote it by $\T_{g,n}$.


Fix a hyperelliptic involution $\sigma$ of $S$. Let $Y_g$ be the set of points of $X_g$ fixed by $\sigma$:
$$ Y_g = X_g^\sigma$$
The set $Y_g$ consists of points $[f: S\to C]$ such that $f\circ\sigma \circ f^{-1}$ is an automorphism of $C$. It is connected and contractible of dimension $2g-1$, since it is biholomorphic to $X_{0, 2g+2}$. Note that $\Delta_g$ is the stabilizer of $Y_g$. For each $m\geq 0$, as an analytic orbifold, the moduli stack $\H_g[m]$ of smooth projective hyperelliptic curves of genus $g$ with a level $m$ structure is the orbifold quotient
$$
\H_g[m] = (\D_g[m]\backslash Y_g)^\orb.
$$
When $m=0$, $\H_g[0]$ is the hyperelliptic Torelli space corresponding to $\sigma$. In fact, the hyperelliptic locus of $X_g$ is the disjoint uniton of translates of $Y_g$. 
Since $Y_g$ is contractible, $\H_g[m]$ is a model of the classifying space of $\Delta_g[m]$. Hence we have an isomorphism 
$$
\pi_1^\orb(\H_g[m]) \cong \Delta_g[m],
$$
where $\pi_1^\orb$ denotes the orbifold fundamental group.
Similarly, for $\Delta_{g,n}$ with $0\leq n\leq 2g+2$,  we denote the orbifold quotient of $Y_g$ by $\Delta_{g,n}$ by
$$
\H_{g,n} = (\Delta_{g,n}\backslash Y_g)^\orb.
$$
Note that $\H_{g, 2g+2} =\H_g[2]$. The orbifold fundamental group $\pi_1^\orb(\H_{g,n})$ is isomorphic to $\Delta_{g,n}$. 
\subsection{Universal family over $\H_{g,n}$}
As an analytic orbifold, $\H_{g,n}$ admits the universal family $\pi_{g,n}:\cC_{\H_{g,n}}\to \H_{g,n}$. It is a proper smooth morphism of orbifolds. Since $n$ Weierstrass points are trivialized, the family $\pi_{g,n}$ admits $n$ distinct sections, which we call Weierstrass sections. Let $x=[C]$ be in $\H_{g,n}$. The fiber of $\pi_{g,n}$ over $x$ is $C$.  Let $q$ be a Weierstrass point of $C$.  Then associated to $\pi_{g,n}$, there is a homotopy exact sequence
\begin{equation}\label{homotopy sequence}
1\to \pi_1(C, q)\to \pi_1^\orb(\cC_{\H_{g,n}}, q)\to \pi_1^\orb(\H_{g,n}, x)\to 1.  
\end{equation}
Denote the fundamental group $\pi_1^\orb(\cC_{\H_{g,n}})$ by $\Delta^\cC_{g,n}$. The monodromy action of $\H_{g,n}$ on $C$ yields natural symplectic representations
$$
\rho_{g,n}:\Delta_{g,n}\to \Sp(H_\Z)
$$
 and 
$$
\rho^\cC_{g,n}:\Delta^\cC_{g,n}\to \Sp(H_\Z).
$$
By Theorem \ref{image of hyp mapping class group}, the images of $\rho_{g,n}$ and $\rho^\cC_{g,n}$ contain $\Sp(H_\Z)[2]$ and hence are Zariski dense in $\Sp(H_\Q)$. 
\section{Relative completions of hyperelliptic mapping class groups}

In this section, we review basics for hyperelliptic mapping class groups and relative completions, while setting up all the notations.


\subsection{Relative completion}
A detailed summary of relative completion can be found in \cite[\S3]{hai4}. 
Let $F$ be a field of characteristic zero. 
Suppose that
\begin{enumerate}[(i)]
    \item $\Gamma$ is a discrete group,
    \item $R$ is a reductive $F$-group, and
    \item $\rho:\Gamma \to R(F)$ is a Zariski-dense homomorphism.
    \end{enumerate}
The relative completion of $\Gamma$ with respect to $\rho$ consists of a proalgebraic $F$-group $\cG$ that is an extension 
\begin{equation}\label{extension for relative comp}
1\to \U\to \cG\to R\to 1
\end{equation}
of $R$ by a prounipotent $F$-group $\U$ and a Zariski-dense homomorphism $\tilde{\rho}:\Gamma \to \cG(F)$ such that the diagram
$$\xymatrix@R=1em@C=2em{
\Gamma \ar[d]_{\tilde{\rho}}\ar[dr]^{\rho}& \\
\cG(F) \ar[r] & R(F)
}
$$
commutes. 
It satisfies the following universal property. Let $G$ be a proalgebraic $F$-group that is an extension 
$$1\to U\to G\to R\to 1$$
of $R$ by a prounipotent $F$-group $U$. If $\rho_G:\Gamma\to G(F)$ is a homomorphism that lifts $\rho$, then there exists a unique homomorphism $\phi: \cG\to G$ such that the diagram
$$\xymatrix@R=1em@C=2em{
\Gamma \ar[d]_{\rho_G}\ar[r]^{\tilde{\rho}}& \cG(F)\ar[d]\ar[ld]_{\phi}\\
G(F) \ar[r] & R(F)
}
$$
commutes.

\subsection{Key properties}
Denote the Lie algebra of $\U$ by $\u$. It is a pro-nilpotent Lie algebra.
Relative completions are to some degree computable, since it is controlled by cohomology. The extension (\ref{extension for relative comp}) splits by a generalization of Levi's Theorem and any two splittings are conjugate by an element of $\U$. Therefore, the Lie algebra $\u$ is an $R$-module and there is an isomorphism
$$ \cG\cong R\ltimes \U \cong R\ltimes \exp\u$$
that is unique up to conjugation by an element of $\U\cong \exp\u$. The following result relates the Lie algebra $\u$ and the group cohomology of $\G$.
\begin{theorem}[{\cite[Thm.~3.8]{hai3}}]\label{iso for rel comp}
For all finite dimensional $R$-modules $V$, 
\begin{enumerate}[(i)]
\item there is a natural isomorphism
$$\Hom_R(H_1(\u), V) \cong H^1(\G, V), $$
and
\item there is a natural injection
$$\Hom_R(H_2(\u), V) \hookrightarrow H^2(\G, V).$$
\end{enumerate}
\end{theorem}
\subsection{Relative completions of hyperelliptic mapping class groups}
Assume that $g \geq 2$. Recall that $G_{g,n}$ is the image of $\rho_{g,n}:\Delta_{g,n}\to \Sp(H_\Z)$. Note also that the hyperelliptic Torelli group $T\Delta_g$ is the kernel of $\rho_g =\rho_{g,0}$
and that $\ker \rho_{g,n} = T\Delta_g$ for each $n =1,\ldots, 2g+2$. The representation 
$\rho^\cC_{g,n}:\Delta^\cC_{g,n}\to\Sp(H_\Z)$ factors through $\rho_{g,n}$, and therefore has the same image $G_{g,n}$. With abuse of notation, we denote the maps restricted to this image also by $\rho_{g,n}$ and $\rho^\cC_{g,n}$.

Set $H =H_\Q$. Denote the relative completion of $\Delta_{g, n}$ with respect to $\rho_{g,n}$ by $\cD_{g, n}$ equipped with  $\tilde\rho_{g,n}:\Delta_{g,n} \to \cD_g(\Q)$.  There is an extension
\begin{equation}\label{rel extension}
1\to\U_{g,n}\to\cD_{g,n}\to \Sp(H)\to 1
\end{equation}
of $\Sp(H)$ by a prounipotent radical $\U_{g,n}$ of $\cD_{g,n}$. Denote the Lie algebras of $\cD_{g,n}$ and $\U_{g,n}$ by $\d_{g,n}$ and $\u_{g,n}$ respectively. When $n=2g+2$, we denote $\cD_{g,n}$, $\U_{g,n}$, $\d_{g,n}$, and $\u_{g,n}$ by $\cD_g[2]$, $\U_g[2]$, $\d_g[2]$, and $\u_g[2]$, respectively. The exact sequence \ref{rel extension} fits in the commutative diagram
$$
\xymatrix@R=2em@C=2em{
1\ar[r]&T\Delta_g\ar[r]\ar[d]_{\tilde\rho_{g,n}|_{T\Delta_g}}&\Delta_{g,n} \ar[r]\ar[d]^{\tilde\rho_{g,n}}&G_{g,n}\ar[r]\ar@{^{(}->}[d]&1\\
1\ar[r]&\U_{g,n}\ar[r]&\cD_{g,n}\ar[r]\ar[r]&\Sp(H)\ar[r]&1.
}
$$
Denote the prounipotent completion of $T\Delta_g$ over $\Q$ by $T\Delta^\un_{g/\Q}$. The right exactness of relative completions \cite[Prop. 3.7]{hai4} gives the exact sequence
$$
T\Delta^\un_{g/\Q}\to \cD_g\to \Sp(H)\to 1.
$$
This implies that the homomorphism $T\Delta^\un_{g/\Q}\to \U_g$ is surjective. Since the image of $T\Delta_g$ in $T\Delta^\un_{g/\Q}$ is Zariski dense, it follows that the image of $\tilde\rho_{g,n}|_{T\Delta_g}$ is Zariski dense.

Similarly, the relative completion of $\Delta^\cC_{g,n}$ with respect to $\rho^\cC_{g,n}:\Delta^\cC_{g,n}\to G_{g,n}$ is denoted by $\cD^\cC_{g,n}$ and $\tilde\rho^\cC_{g,n}:\Delta^\cC_{g,n}\to \cD^\cC_{g,n}(\Q)$ and its prounipotent radical by $\U^\cC_{g,n}$. Denote their Lie algebras by $\d^\cC_{g,n}$ and $\u^\cC_{g,n}$ respectively.

Denote the unipotent completion of $\pi_1(C, q)$ over $\Q$ and its Lie algebra by $\cP$ and $\p$, respectively. The center freeness of $\cP$ imply that the sequence \ref{homotopy sequence} gives exact sequences
\begin{equation}\label{rel homotopy seq}
1\to \cP\to \cD^\cC_{g,n}\to \cD_{g,n}\to 1
\end{equation}
 and 
\begin{equation}\label{rel lie homotopy seq}
0\to \p\to \d^\cC_{g,n}\to \d_{g,n}\to 0.
\end{equation}

\subsection{Mixed Hodge structures on the relative completions}
Let $V$ be the dual of $R^1\pi_{g,n\ast}\Z$. It is a polarized  variation of Hodge structure of weight $-1$. Its monodromy representation can be identified with $\rho_{g,n}$. By the main theorem of Hain in \cite{hai2}, the Lie algebras of the relative completions $\d_{g,n}$ and $\d^\cC_{g,n}$ admit natural $\Q$-MHSs, where their Lie brackets are morphisms of MHSs. The Lie algebra $\p$ admits a natural $\Q$-MHS, being the Lie algebra of the unipotent completion of $\pi_1(C, q)$. On the other hand, it also admits a $\Q$-MHS as the kernel of the surjection $\d^\cC_{g,n}\to \d_{g,n}$. By the naturality properties of the MHS of relative completions \cite[Thm.~13.12]{hai2}, these MHSs on $\p$ are equal. 
Denote the Lie algebra of $\Sp(H)$ by $\mathfrak{s}$. The basic properties of the weight filtrations of these Lie algebras follow from \cite[Cor.~13.2]{hai2} and listed below.

\begin{proposition}
The weight filtrations $W_\bullet\d_{g,n}$, $W_\bullet\d^\cC_{g,n}$, and $W_\bullet\p$ satisfy the properties:
\begin{enumerate}
\item 
$$
W_{-1}\d_{g,n} =\u_{g,n}=W_{-1}\u_{g,n},\,\,\,\,  W_{-1}\d^\cC_{g,n} =\u^\cC_{g,n} =W_{-1}\u^\cC_{g,n},\hspace{.1in}\text{and}\hspace{.1in}\p=W_{-1}\p,
$$
\item 
$$
\Gr^W_0\d_{g,n} =\Gr^W_0\d^\cC_{g,n} =\mathfrak{s},
$$
and 
\item 
for $m\geq1$, each graded quotients $\Gr^W_{-m}\d_{g,n}$, $\Gr^W_{-m}\d^\cC_{g,n}$, and $\Gr^W_{-m}\p$ are $\Sp(H)$-modules, and the induced maps 
$$\Gr^W_\bullet\p\to \Gr^W_\bullet\d^\cC_{g,n} \text{ and } \Gr^W_\bullet\d^\cC_{g,n}\to \Gr^W_\bullet\d_{g,n}
$$ are $\Sp(H)$-equivariant graded Lie algebra homomorphisms. 

\end{enumerate}
\end{proposition}

\section{Hyperelliptic Johnson homomorphisms}
Fix a Weierstrass point $q$ in $S$. Denote the fundamental group $\pi_1(S, q)$ by $\Pi$, and denote the pronilpotent Lie algebra of the unipotent completion (Malcev completion, see \cite{mal}) of $\Pi$ over $\Q$ by $\p$ as well (The unipotent completions of $\pi_1(C)$ and $\pi_1(S)$ are isomorphic). 
  For a group $G$, denote the lower central series of $G$ by $L^\bullet G$, where $L^1G=G$ and $L^kG = [L^{k-1}G, G]$ for $k \geq 2$. For each $k\geq 1$, denote the associated graded quotient $L^kG/L^{k+1}G$ by $\Gr^L_kG$ and the quotient $G/L^kG$ by $N_kG$. There is an exact sequence
$$ 1\to \Gr^L_kG \to N_{k+1}G\to N_{k}G\to 1.$$
When $k=2$, we have the exact sequence
$$ 1\to \Gr^L_2G\to N_3G\to H_1(G)\to 1.$$
The mapping class group $\G_{g, 1}$ acts $N_k\Pi$ and so there is a homomorphism 
$$\rho_k: \G_{g,1}\to \Aut(N_k\Pi).$$
The action of $\G_{g,1}$ on $\Gr^L_k\Pi$ factors through the natural representation $\rho:\G_{g,1}\to \Sp(H_\Z)$. 
Note that when $k=2$, $N_2\Pi=H_1(S, \Z)=H_\Z$, $\rho_2$ is the homomorphism $\rho$, and $\ker \rho_2 = T_{g,1}$. The Johnson homomorphism 
$$\tau: \ker \rho_2 = T_{g,1}\to \Hom(H_\Z, \Gr^L_2\Pi)$$
is defined by setting $\phi\mapsto (u\mapsto \phi(\tilde{u})\tilde{u}^{-1} \in \Gr^L_2\Pi)$, where $\tilde{u}$ is any lift of $u$ in $N_3\Pi$. It is a $\G_{g,1}$-equivariant homomorphism, and the induced homomorphism $H_1(T_{g,1})\to \Hom(H_\Z, \Gr^L_2\Pi)$ is $\Sp(H_\Z)$-equivariant. It is well known that as an $\Sp(H_\Z)$-module, $\Gr^L_2\Pi$ is isomorphic to $\Lambda^2H_\Z/\langle \theta \rangle$. In fact, Johnson proved in \cite{joh1, joh2} that the image of $\tau$ is contained in $\Lambda^3H_\Z \subset \Hom(H_\Z,\Lambda^2H_\Z/\langle \theta \rangle)$ and $\tau$ induces an isomorphism $H_1(T_{g,1}) \to \Lambda^3H_\Z$ mod 2-torsion. \\
\indent In the case of the hyperelliptic mapping class group $\Delta_{g,1}$, the corresponding Johnson homomorphism is trivial as follows.
Since the hyperelliptic involution $\sigma$ commutes with the elements of $T\Delta_g$, it acts trivially on $T\Delta_g$. On the other hand, $\sigma$ acts as $-\id$ on $\Hom(H_\Z, \Gr^L_2\Pi)$. The triviality of the homomorphism implies that elements of $T\Delta_g$ act trivially on $N_3\Pi$. From the exact sequence
$$ 1\to \Gr^L_3\Pi \to N_{4}\Pi\to N_{3}\Pi\to 1,$$
we obtain the hyperelliptic Johnson homomorphism
$$\tau^\hyp_q: T\Delta_g\to \Hom(H_\Z, \Gr^L_3\Pi)\,\,\,\,\phi\mapsto (u\mapsto \phi(\tilde{u})\tilde{u}^{-1}\in \Gr^L_3\Pi),$$
where $\tilde{u}$ is any lift of $u$ in $N_4\Pi$. The homomorphism $\tau^\hyp_q$ is $\Delta_{g,1}$-equivariant.

\subsection{Key $\Sp(H)$-representations} Suppose that $g\geq 2$. Set $H = H_\Q$ and fix a symplectic basis $a_1, b_1, \ldots, a_g, b_g$ of $H$. 
The isomorphism class of each finite dimensional irreducible representation of $\Sp(H)$ can be indexed by a partition $\lambda$ of a nonnegative integer $n$ into less than $g$ parts:
$$n = \lambda_1 + \lambda_2 +\cdots + \lambda_l\,\,\text{with $l\leq g$ and } \lambda_1\geq\lambda_2\geq\cdots\geq\lambda_l.$$
Denote the irreducible $\Sp(H)$-representation whose isomorphism class corresponds to a partition $\lambda$ by $V_\lambda$. Let $\theta =\sum_{i=1}^g a_i\wedge b_i$. Then for example, we have
$$V_{[1]}\cong H, \hspace{.1in}V_{[1^2]}\cong \Lambda^2H/\langle \theta \rangle, \text{ and }V_{[1^3]}\cong \Lambda^3H/\theta \wedge H.$$

\subsection{The Lie algebras $\p$ and $\Der\p$ as $\Sp(H)$-representations}
Since $H_1(\p)$ is pure of weight $-1$, it follows that the natural weight filtration on $\p$ coincides with the lower central series, that is, $W_{-m}\p = L^m\p$ for each $m\geq 1$.
Each associated graded quotient $\Gr^W_{-m}\p=W_{-m}\p/W_{-m-1}\p$ denoted by $\p(-m)$ is a finite dimensional $\Sp(H)$-module. 
  Let $V$ be a vector space over a field. Denote the free Lie algebra generated by $V$ by $\L(V)$. It is the direct sum $\oplus_{k\geq 1} \L_k(V)$ of components $\L_k(V)$ of bracket length $k$.   It is well known that the graded Lie algebra $\Gr^W_\bullet\p$ has a minimal presentation 
$$\Gr^W_\bullet \p \cong \L(H)/\langle \theta \rangle.$$

The following result shows some of the $\Sp(H)$ representations appearing in $\Gr^W_\bullet\p$ for small values of $m$.
\begin{proposition}[{\cite[Prop.~8.4, Cor.~8.5]{hai3}}] For all $g\geq 2$, the irreducible decomposition of $\p(-m)$ as an $\Sp(H)$-representation for $1\leq m \leq 3$ is given by
$$\p(-m) =
\begin{cases}V_{[1]} &\text{ for } m=1\\
V_{[1^2]} & \text{ for } m=2\\
V_{[2+1]} & \text{ for } m=3.

\end{cases}
$$
\end{proposition}

The derivation algebra $\Der \p$ is also a MHS induced by that of $\p$.  Another key object we consider is $\Gr^W_{-2}\Der\p$.
The derivation algebra $\Der \L(H)$ of $\L(H)$ is given 
$$
\Der\L(H) = \oplus_{k\geq 0} \Hom(H, \L_k(H)).
$$
Furthermore, the derivation $\Gr^W_\bullet\Der\p = \Der\Gr^W_\bullet\p$ is an $\Sp(H)$-submodule of $\Der\L(H)$ and is given by
$$
\Gr^W_\bullet\Der\p =\oplus_{m\geq 0}\Der_{-m}\p,
$$
where $\Der_{-m}\p$ is the $\Sp(H)$-submodule of $\Hom(H, \L_{m+1}(H))$ consisting of derivations that annihilate $\theta$. By \cite[Prop.~9.1]{hai3}, in the representation ring of $\Sp(H)$,  we have
$$
\Der_{-m}\p = \p(-1)\otimes \p(-1-m) - \p(-2-m).
$$
In this paper, the weight $-2$ component $\Der_{-2}\p$ plays an essential role. 
\begin{proposition}[{\cite[Prop.~9.1, Cor.~9.2]{hai3}}]\label{derp w -2} For all $g\geq 2$, the irreducible decomposition of $\Der_{-2}\p$ as an $\Sp(H)$-module is given by
$$\Der_{-2}\p =
V_{[2^2]} + V_{[1^2]}.
$$

\end{proposition}

\subsection{Projection of $\Der_{-2}\p$ onto $V_{[1^2]}$} Here we will explain how to identify the $V_{[1^2]}$-component of $\Der_{-2}\p$, using the projection used in \cite[Cor.5.1]{colombo} and give an explicit formula.

Define an $\Sp(H)$-equivariant map $\phi: S^2\Lambda^2H \to \Hom(H, \L_3(H))$ by
$$
\phi: (u_1\wedge v_1)(u_2\wedge v_2)\mapsto $$
$$
x \mapsto 
\langle u_1, x\rangle[v_1, [u_2, v_2]] + \langle v_1, x\rangle[[u_2, v_2], u_1] +\langle u_2, x\rangle [v_2,[u_1, v_1]] + \langle v_2, x\rangle[[u_1, v_1], u_2],
$$
where $\langle\cdot, \cdot\rangle:\Lambda^2H\to \Z$ is the algebraic intersection paring. 

 The following result will be useful to describe the image of a Dehn twist under a hyperelliptic Johnson homomorphism. An easy computation gives
\begin{lemma}\label{theta_I image}
For $I\subset \{1, \ldots, g\}$, set $H_I = \mathrm{Span}\{a_i, b_i| i\in I\}$ and \\
$H^c_I =\mathrm{Span}\{a_i, b_i|i \not\in I\}$. Let $\theta_I =\sum_{i\in I}a_i\wedge b_i $. Then
$$\phi(\theta_I^2)(x) = \begin{cases} 0 & x \in H^c_I \\
-2[\theta_I, x] & x \in H_I \end{cases}.
$$
\end{lemma}
Furthermore, we have
\begin{lemma}
The image of $\phi$ is in $\Der_{-2}\p$. 
\end{lemma}
\begin{proof}
Let $\theta =\sum_{i=1}^g a_i\wedge b_i$. Then for each $(u_1\wedge v_1)(u_2\wedge v_2)$, we have 
\begin{align*} \phi((u_1\wedge v_1)(u_2\wedge v_2))(\theta) & = \sum_{i=1}^g[-\langle u_1, a_i\rangle b_i +\langle u_1, b_i\rangle a_i, [v_1, [u_2, v_2]]] \\
&+ [-\langle v_1, a_i\rangle b_i +\langle v_1, b_i\rangle a_i, [[u_2, v_2], u_1]]\\ &+[-\langle u_2, a_i\rangle b_i +\langle u_2, b_i\rangle a_i, [v_2, [u_1, v_1]]]\\
& +[-\langle v_2, a_i\rangle b_i +\langle v_2, b_i\rangle a_i, [[u_1, v_1], u_2]]\\
& = [u_1, [v_1, [u_2, v_2]]] + [v_1, [[u_2, v_2], u_1]]\\
& + [u_2, [v_2, [u_1, v_1]]]+ [v_2, [[u_1, v_1], u_2]]\\
& = - [[u_2, v_2], [u_1, v_1]] - [[u_1, v_1], [u_2, v_2]]\\
& = [[u_1, v_1], [u_2, v_2]]-[[u_1, v_1], [u_2, v_2]] =0.
\end{align*}
Hence, our claim follows. 
\end{proof}
  In fact, it is not difficult to see that $\phi$ is a surjection onto $\Der_{-2}\p$ by Schur's Lemma. We define the explicit projection of $\Der_{-2}\p$ onto the copy of $V_{[1^2]}$ as follows.
First, define an $\Sp(H)$-equivariant map  $p_H:\otimes^3H\to H$ by 
$$
u\otimes v \otimes w \mapsto \langle u, v\rangle w. 
$$
Note that the map $p_H$ is the dual of the injection $\theta\otimes\cdot : H \hookrightarrow \otimes^3H$ and that the composition $p_H\circ (\theta\otimes\cdot) = 2g\id_H$.
Secondly, define an $\Sp(H)$-equivariant map $p_{\Lambda^2H}: \Hom(H, \otimes^3H)\to \Lambda^2H$ by
$$
\gamma \mapsto \sum_{i=1}^g a_i\wedge p_H\gamma(b_i)- b_i\wedge p_H\gamma(a_i).
$$
Now, we have the identification $\L_3(H) = (\Lambda^2H\otimes H)/\Lambda^3 H$ and then using the inclusion $\Lambda^2 H \to \otimes^2H$, $u\wedge v\mapsto u\otimes v - v\otimes u$, we get the inclusion $\L_3(H)\to \otimes^3H$. Therefore, we have the  $\Sp(H)$-equivariant projection $ \Der_{-2}\p\to \Lambda^2H$ given by the composition
$$\Der_{-2}\p \hookrightarrow \Hom(H, \L_3(H))\hookrightarrow \Hom(H, \otimes^3H)\overset{p_{\Lambda^2H}}\to \Lambda^2H,
$$
which we denote by $\pi_{\Lambda^2H}$.
The following result is useful when computing the $V_{[1^2]}$-component in $\Der_{-2}\p$ of the image of an element of $T\Delta_g$ under $\tau^\hyp$. 
\begin{lemma}\label{proj formula onto [1^2]}
The composition $\pi_{\Lambda^2H}\circ \phi: S^2\Lambda^2H \to \Lambda^2H$ is given by
$$
(u_1\wedge v_1)(u_2\wedge v_2) \mapsto 4[ \langle u_1, v_1\rangle v_2\wedge u_2 + \langle v_2, u_2\rangle u_1\wedge v_1] +
$$
$$
2[\langle u_1, v_2\rangle v_1\wedge u_2 + \langle v_1, u_2\rangle u_1\wedge v_2 + \langle u_1, u_2\rangle v_2\wedge v_1 + \langle v_2, v_1\rangle u_1\wedge u_2].
$$
\end{lemma}
\begin{proof}
A direct computation suffices. 
\end{proof}
Denote the projection $\Lambda^2H \to \Lambda^2H/\langle \theta \rangle$ by $\tilde\theta$.  On the other hand, we may view $V_{[1^2]}$ as a submodule of $\Lambda^2 H$ and there is the $\Sp(H)$-projection $\hat\theta: \Lambda^2 H \to V_{[1^2]}$ given by $u\wedge v \mapsto u\wedge v - \frac{\langle u,v\rangle }{g}\theta$. Denote the composition $\hat\theta\circ \pi_{\Lambda^2H}\circ \phi$ by $\hat\pi$. Denote the map $V_{[1^2]} \to S^2\Lambda^2H$ given by multiplication by $\theta$ by $j_\theta$. An easy computation together with Lemma \ref{proj formula onto [1^2]} gives

\begin{lemma}\label{inclusion of [1^2]}
The composition
$$
V_{[1^2]} \overset{j_\theta}\to S^2\Lambda^2H \overset{\hat\pi}\to V_{[1^2]}
$$
is given by $-4(g+1)$ times the identity map $\id_{V_{[1^2]}}$.
\end{lemma}
 
\begin{lemma}\label{tambo proj}
With notation as above, for each $x\in S^2\Lambda^2H$, the vector 
$$
x - \frac{1}{-4(g+1)} (j_\theta \circ \hat{\pi}) (x)
$$
lands in the $V_{[2^2]}$-component of $\Der_{-2}\p$ via $\phi$.
\end{lemma}
\begin{proof}
Let $s =\phi\left(x - \frac{1}{-4(g+1)} (j_\theta \circ \hat{\pi}) (x)\right)$. By Lemma \ref{inclusion of [1^2]}, $(\hat\theta\circ\pi_{\Lambda^2H})(s)= 0$, and hence by Proposition \ref{derp w -2} and Schur's Lemma,  $s$ is in the $V_{[2^2]}$-component of $\Der_{-2}\p$. 
\end{proof}

\subsection{Nontriviality of $\tau^\hyp$}\label{johnson hyp nonzero}
Suppose that $g\geq 2$. 
Let $c_i$ be a separating simple closed curve in $S$ that divides $S$ into two subsurfaces $S'_i$ and $S''_i$ of genus $i$ and $g-i$, respectively. Fix a symplectic basis $a_1, b_1, \ldots, a_g, b_g$ for $H_\Z$ such that for each $i = 1, \ldots, g-1$, the sets $\{a_l, b_l|1\leq l\leq i\}$ and $\{a_l, b_l|i+1\leq l \leq g\}$ form symplectic bases for $H_1(S'_i)$ and $H_1(S''_i)$, respectively. Let $\theta'_i =\sum_{l=1}^i a_l\wedge b_l$ and $\theta''_i = \sum_{l=i+1}^g a_l\wedge b_l$. Then we have $\theta =\theta_i' + \theta_i''$. Denote the isotopy class of the Dehn twist along  $c_i$ by $d_i$. It is an element of the hyperelliptic Torelli group $T\Delta_g$. For simplicity, assume that the Weierstrass point $q$ lies in the subsurface $S'_i$. 
Lemma \ref{theta_I image} implies that we have

$$\phi((\theta''_i)^2)(x) = \begin{cases} 0 & x \in H_1(S'_i) \\
-2[\theta''_i, x] & x \in H_1(S''_i) \end{cases}.
$$
\begin{figure}[h]
    \centering
    \includegraphics[scale=0.6]{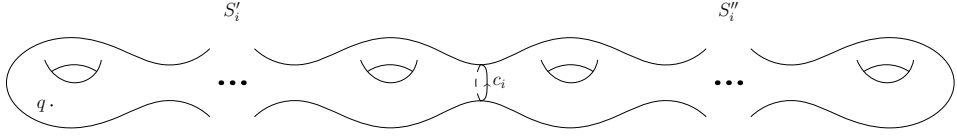}
    \caption{\textcolor{black}{Subsurfaces $S'_i$ and $S''_i$ separated by $c_i$}}
    \label{fig1}
\end{figure}
It then follows from the construction of $\tau^\hyp_q$ with the base point $q$ that we have 
\begin{proposition}[{\cite[Prop.~6.2]{wat}}]\label{dehn twist image}
If $g\geq 2$, then 

$$ \tau^\hyp_q(d_i) = \frac{1}{2}\phi((\theta''_i)^2).$$

\end{proposition}
The following result of Brendle, Margalit, and Putman is a key to understand the hyperelliptic Johnson homomorphisms. A separating curve $d$ is said to be symmetric if $\sigma(d) = d$. 

\begin{theorem}[{\cite[Thm.~A]{bmp}}]\label{hyp torelli generators}
For $g\geq 0$, the group $T\Delta_g$ is generated by Dehn twists about symmetric separating curves.
\end{theorem}

As immediate consequences of Proposition \ref{dehn twist image} and Theorem \ref{hyp torelli generators}, we have 
\begin{corollary}
If $g\geq 2$, $\tau^\hyp_q$ is nontrivial and the image of $\tau^\hyp_q$ is contained in $\Der_{-2}\p$. 
\end{corollary}

\subsection{The cohomology class of $\tau^\hyp_q$ }\label{weierstrass class}
Here we associate a cohomology class to $\tau^\hyp_q$ via the relative completion as follows. Assume that $n\geq 1$ and $\Delta_{g,n}$ fixes $q$. Consider the homotopy exact sequence 
$$
\xymatrix{
1\ar[r]& \pi_1(C, q)\ar[r] &\pi_1^\orb(\cC_{\H_{g,n}}, q)\ar[r]_{\pi_{g,n\ast}} &\pi_1^\orb(\H_{g,n}, x)\ar@/_1pc/[l]_{s_{q\ast}}\ar[r]& 1.
}
$$
 The Weierstrass section $s_q$ of the universal curve $\pi_{g,n}:\cC_{\H_{g,n}} \to \H_{g,n}$ given by $q$ induces a section $s_{q\ast}$ of $\pi_{g,n\ast}$.
Taking the relative completion of $\D^\cC_{g,n}$ and $\D_{g,n}$ 
produces the exact sequence \ref{rel homotopy seq}, which yields the exact sequence of prounipotent groups over $\Q$
$$
\xymatrix{
1\ar[r]& \cP\ar[r]& \U^\cC_{g,n}\ar[r]& \U_{g,n}\ar@/_1pc/[l]_{\tilde s_{q}}\ar[r]& 1.
}
$$
By the universal property of relative completions, $s_{q\ast}$ induces a section $\tilde s_q$ of $\cD^\cC_{g,n}\to \cD_{g,n}$, which restricts to a section  of $\U^\cC_{g,n}\to \U_{g,n}$. We denote this restriction by $\tilde s_q$ as well.  
Applying the $\mathrm{log}$ map to the sequence, we obtain the exact sequence of pronilpotent Lie algebras
$$
\xymatrix{
0\ar[r]& \p\ar[r]& \u^\cC_{g,n}\ar[r]& \u_{g,n}\ar@/_1pc/[l]_{d\tilde s_{q}}\ar[r]& 0.
}
$$
 The section $\tilde s_q$ induces a Lie algebra section $d\tilde s_q$ of $\u^\cC_{g,n}\to \u_{g,n}$. Therefore, the Lie algebra $\u_{g,n}$ acts on $\p$ via the adjoint action of $\u^\cC_{g,n}$ on $\p$, and hence we have the adjoint map $\mathrm{adj}_q: \u_{g,n}\to \Der\p$. Since $\adj_q$ preservers the weight filtrations, we obtain an $\Sp(H)$ equivariant graded Lie algebra homomorphism 
 $$
 \Gr^W_\bullet \adj_q: \Gr^W_\bullet\u_{g,n}\to \Gr^W_\bullet \Der\p. 
 $$
 The proof of \cite[Prop.~7.7]{wat} implies that $H_1(\u_{g,n})$ is pure of weight $-2$. Therefore, we have 
 $$\Gr^W_{-2}\u_{g,n} = \Gr^W_{-2}H_1(\u_{g,n}) = H_1(\u_{g,n}),$$
 and  so $\Gr^W_{-2}\adj_q$ can be expressed as an $\Sp(H)$-equivariant map
 $$
 \Gr^W_{-2}\adj_q: H_1(\u_{g,n})\to \Der_{-2}\p.
 $$

 On the other hand,  recall that the restriction of the relative completion $\tilde\rho_{g,n}: \Delta_{g,n}\to \cG_{g,n}$ to $T\Delta_g$ induces the map $T\Delta_g\to \U_{g,n}$, whose image is Zariski dense in $\U_{g,n}$. Composing with the log map $\U_{g,n}\to \u_{g,n}$ and $\u_{g,n}\to H_1(\u_{g,n})$, we obtain the map $r_{g,n}: T\Delta_g\to H_1(\u_{g,n})$.  Now, since the adjoint map is induced by the conjugation action of $\Delta^\cC_{g,n}$ on $\pi_1(C, q)$, the construction of the hyperelliptic Johnson homomorphism implies that there is a commutative diagram 

$$\xymatrix@R=1em@C=3em{
T\D_g\ar[d]_{r_{g,n}}\ar[rd]^{\tau^\hyp_q}\\
	H_1(\u_{g,n})\ar[r]_{\Gr^W_{-2}\adj_q}&\Der_{-2}\p.
}
$$
 Denote the composition 
$$
\tilde\theta\circ \pi_{\Lambda^2H}\circ \tau^\hyp_q: T\Delta_g \to \Lambda^2H/\langle\theta\rangle = V_{[1^2]}
$$ 
by $\tilde\tau^\hyp_q$ and  the composition 
$$\tilde\theta\circ \pi_{\Lambda^2H}\circ \Gr^W_{-2}\adj_q : H_1(\u_{g,n}) \to \Lambda^2H/\langle \theta \rangle$$
 by $\tilde\tau^\adj_q$. Since the map $r_{g,n}$ has a Zariski dense image, it follows that the homomorphism $\tilde\tau^\adj_q$  is a unique $\Sp(H)$-equivariant map that makes the diagram
 $$
 \xymatrix@R=1em@C=2em{
T\Delta_g\ar[d]_{r_{g,n}}\ar[dr]^{\tilde\tau^\hyp_q}&\\
H_1(\u_{g,n})\ar[r]_{\tilde\tau^\adj_q}&\Lambda^2H/\langle \theta\rangle
  }
 $$
 commute. Now the homomorphism $\tilde\tau^\adj_q$ corresponds to a class in $H^1(\Delta_{g,n}, V_{[1^2]})$, denoted by $[q]$, via the natural isomorphism 
 $$H^1(\Delta_{g,n}, V_{[1^2]})\cong \Hom_{\Sp(H)}(H_1(\u_{g,n}), V_{[1^2]})$$
 induced by relative completion in Theorem \ref{iso for rel comp}. Since $\Delta_g[2]$ is a subgroup of $\Delta_{g,n}$, there is a natural homomorphism $H^1(\D_{g,n}, V_{[1^2]})\to H^1(\D_g[2], V_{[1^2]})$.  It is easy to see that the image of $[q]$ in $H^1(\Delta_g[2], V_{[1^2]})$ is equal to the class of $q$ produced by the above construction  when $n = 2g+2$. \\



\section{Normal functions and Ceresa cycles}

Normal functions are holomorphic sections of families of intermediate Jacobians that satisfy certain asymptotic conditions. This notion naturally arise when studying families of algebraic varieties. In this section, we recall the construction of the normal function associated to a family of homologically trivial algebraic cycles in a family of smooth projective varieties. Then we review Ceresa cycles and how their associated normal function relates to the Johnson homomorphism. More details can be found in Hain \cite{hain,hain2}.

\subsection{Intermediate Jacobians}\label{intermediate jac}
Suppose that $X$ is a smooth projective variety and that $Z$ is an algebraic $d$-cycle in $X$. We have an exact sequence of mixed Hodge structures
$$0\to H_{2d+1}(X)\to H_{2d+1}(X,Z)\to H_{2d}(Z)\to H_{2d}(X).$$
The class of the cycle $Z$ defines a morphism of mixed Hodge structures
$$c_Z:\Z(d)\to H_{2d}(Z).$$
If $Z$ is homologically trivial, we can pull back the previous sequence along $c_Z$ to obtain an extension
$$0\to H_{2d+1}(X)\to E_Z\to\Z(d)\to 0$$
in the category $\MHS$ of mixed Hodge structures. Tensoring with $\Z(-d)$ gives an extension
$$0\to H_{2d+1}(X,\Z(-d))\to E_Z(-d)\to\Z\to 0$$
and thus a class $e_Z$ in
$$\Ext^1_\MHS(\Z,H_{2d+1}(X,\Z(-d))).$$
Note that $H_{2d+1}(X)$ has weight $-(2d+1)$, and thus $H_{2d+1}(X,\Z(-d))$ has weight $-1$.

For a mixed Hodge structure $V$ whose weights are all negative, there is a natural isomorphism
$$J(V)\cong\Ext^1_\MHS(\Z,V)$$
where $$J(V):=\frac{V_\C}{F^0V_\C+V_\Z}.$$
In general, by Carlson \cite{carlson}, we have 
$$\Ext^1_\MHS(B,A)\cong J(\Hom(B,A))$$
where $\Ext^1_\MHS(B,A)$ is the set of congruence classes of extensions of $B$ by $A$ for separated mixed Hodge structures $A$ and $B$, i.e. the highest weight of $A$ is less than the lowest weight of $B$.

The class $e_Z$ of a homologically trivial $d$-cycle $Z$ in $X$ can thus be viewed as a class 
$$e_Z\in J(H_{2d+1}(X,\Z(-d))),$$
which, in turn, can be viewed as a class in the $d$-th intermediate Jacobian
$$J_d(X):=(F^{d+1}H^{2d+1}(X))^*/H_{2d+1}(X ,\Z)$$
as $J_d(X)\cong J(H_{2d+1}(X,\Z(-d)))$. This class can be described explicitly by Griffiths' generalization of the Abel-Jacobi construction as follows. Write $Z=\partial\Gamma$, where $\Gamma$ is a topological $(2d+1)$-chain. Each class in $F^{d+1}H^{2d+1}(X)$ can be represented by a closed form in the Hodge filtration $F^{d+1}$ of the de Rham complex of $X$. By Stokes' theorem, integrating these representatives over $\Gamma$ gives a well defined functional
$$\int_\Gamma:F^{d+1}H^{2d+1}(X)\to\C.$$
The choice of $\Gamma$ is unique up to a topological $(2d+1)$-cycle. So $\int_\Gamma$ determines a point in $J_d(X)$ that corresponds to $e_Z$.

\subsection{Normal functions}
Suppose that $T$ is a smooth variety and that $\V\to T$ is a variation of mixed Hodge structures of negative weights over $T$. Let $\cV$ be the corresponding bundle whose fiber over $t\in T$ is $V_t$. Denote by $\J(\cV)$ the bundle over $T$ whose fiber over $t$ is 
$$J(V_t)\cong\Ext^1_\MHS(\Z,V_t).$$

\begin{definition}\label{nf}
A holomorphic section $s:T\to\J(\cV)$ of $\J(\cV)\to T$ is a \emph{normal function} if it defines an extension 
$$0\to\V\to\E\to\Z_T\to 0$$
in the category of admissible variations of mixed Hodge structure over $T$.
\end{definition}

\begin{remark}\label{vmhs}
Admissibility characterizes good variations in the sense of Steenbrink--Zucker\cite{sz} and Saito\cite{saito}. It is satisfied in the geometric situations, Guill\'en et al\cite{GNP} and Hain\cite{hain3}. An important characterization is the existence of a relative weight filtration at the infinity, which amounts to certain asymptotic conditions.
\end{remark}

\begin{example}
Families of homologically trivial algebraic cycles give rise to such extensions of variations of mixed Hodge structures and thus normal functions.
Suppose that $\X\to T$ is a family of smooth projective varieties over a smooth base $T$. Suppose that $\cZ$ is an algebraic cycle in $\X$, which is proper over $T$ of relative dimension $d$. Denote the fibers of $\X$ and $\cZ$ over $t\in T$ by $X_t$ and $Z_t$ respectively. Suppose that $Z_t$ is homologically trivial in $X_t$ for all $t$.

The Hodge structures $H_{2d+1}(X_t,\Z(-d))$ form a variation of Hodge structure $\V$ over $T$ of weight $-1$. The intermediate Jacobians $J_d(X_t)\cong J(H_{2d+1}(X_t,\Z(-d)))$ form the relative intermediate Jacobian
$$\J_d\to T.$$
The family of cycles $\cZ$ defines a section of this bundle. This is called \emph{the normal function of the cycle $\cZ$}.
\end{example}

\subsection{Ceresa cycles and their associated normal functions}
Let $C$ be a compact Riemann surface of genus $g$, and $JC$ its Jacobian. When $g\ge 1$, for each $x\in C$, the Abel-Jacobi map
\begin{align*}
    \nu_x:C &\to JC \\
    y &\mapsto y-x
\end{align*}
is an embedding. Denote the image by $C_x$, which is an algebraic 1-cycle in $JC$. There is an involution on $JC$ by $D\mapsto -D$. Denote the image of $C_x$ under this involution by $C_x^-$. Since the involution acts trivially on $H_2(JC)$, the algebraic 1-cycle
$$Z_{C,x}:=C_x-C_x^-$$
is homologically trivial. By \S\ref{intermediate jac}, this $Z_{C,x}$ determines a point in the intermediate Jacobian 
$$e_{C,x}\in J_1(JC)\cong J(H_3(JC,\Z(-1))).$$
The primitive decomposition 
$$H_3(JC,\Q)=H_1(JC,\Q)\oplus PH_3(JC,\Q)$$
is the decomposition of $H_3(JC,\Q)$ into irreducible \textcolor{black}{$\Sp(H_1(C))$}-modules, the highest weights of the pieces being $\lambda_1$ and $\lambda_3$. Pontrjagin product with the class of $C$ induces a homomorphism 
$$\Phi:JC\to J_1(JC).$$
Denote the cokernel of $\Phi$ by $JQ(JC)$. By \cite[Prop. 6.1]{hain}, we have
$$e_{C,x}-e_{C,y}=\Phi(x-y).$$
It follows that the image of $e_{C,x}$ in $JQ(JC)$ is independent of $x$. The image will be denoted by $e_C$.

Now suppose that the genus $g\ge 3$. Denote by $\J_1$ and $\J_{1\prim}$ the bundles over $\M_{g,n}$ whose fiber over $[C;\{x_1,\cdots,x_n\}]$ is $J_1(JC)$ and $JQ(JC)$ respectively. 

We construct sections ${e}_{g,1}:[C,x]\mapsto e_{C,x}$ and $e_g:[C]\mapsto e_C$ of $\J_1\to\M_{g,1}$ and $\J_{1\prim}\to\M_g$ respectively.
\begin{center}
    \begin{tikzcd}
    \J_1 \ar[d] & \J_{1\prim} \ar[d] \\
    \M_{g,1} \ar[u, bend left=40, "{e}_{g,1}"] & \M_g \ar[u, bend right=40, "e_g"']
    \end{tikzcd}
\end{center}

By Definition \ref{nf} and Remark \ref{vmhs}, we have the following result by construction.
\begin{theorem}
The sections ${e}_{g,1}$ and $e_g$ are normal functions. 
\end{theorem}

Fix base points $[C,x]\in\M_{g,1}$ and $[C]\in\M_g$ respectively. So $H:=H_1(C,\Z)$ is fixed. Taking fundamental groups, the normal functions ${e}_{g,1}$ and $e_g$ induce $\Sp(H)$-module homomorphisms
$$\xi_{g,1}: H_1(T_{g,1},\Z)\to H_1(J_1(JC),\Z)\cong\Lambda^3 H_1(C,\Z)$$
and 
$$\xi_{g}: H_1(T_{g},\Z)\to H_1(JQ(JC),\Z)\cong\Lambda^3 H_1(C,\Z)/H_1(C,\Z)$$
respectively. The following result is shown in Hain \cite[Prop. 6.3]{hain}.
\begin{theorem}
The map $\xi_{g,n}$ is twice the Johnson homomorphism $\tau_{g,n}$ for $n=0,1$.
\end{theorem}

\section{Collino cycles and their associated normal functions}
In this section, we review Colombo's results on Collino cycles \cite{colombo}. Since these can be viewed as degenerations of Ceresa cycles, they give rise to elements in higher Chow groups of the Jacobian $JC$ of a hyperelliptic curve $C$ \cite[1.1]{collino}. Their regulator images, defined in \cite{beilinson, bloch}, can be expressed in terms of iterated integrals. They are thus identified with some specific extensions constructed from the fundamental group of $C$. As in the case of Ceresa cycles, we can construct the normal functions associated to Collino cycles and compute their induced monodromies.

\subsection{Regulators}
Let $X$ be a smooth projective variety of dimension $n$. An element in the first higher Chow group $CH^n(X,1)$ is defined by 
$$A:=\sum_i (C_i,f_i)$$
where $C_i$ is an irreducible curve on $X$ and $f_i$ a rational function on $C_i$ such that
$$\sum_i [div(f_i)]=0.$$
Let $\gamma_i:=f_i^{-1}([0,\infty])$. The above condition tells us that 
$$\eta:=\sum_i\gamma_i$$
is a loop. In fact, it is homologically trivial, so that
$$\eta=\partial D$$ 
where $D$ is a 2-chain. Then the regulator map
\begin{align*}
\reg: CH^n(X,1)&\to \textcolor{black}{I_2(X)}:=(F^1H^2(X))^*/H_2(X,\Z(1)) \\
A&\mapsto\reg(A)
\end{align*}
where $\reg(A)$ denotes the current, i.e. linear functional on $F^1H^2(X)$, taking
$$\alpha\mapsto\sum_i\int_{C_i-\gamma_i}\log (f_i)\alpha+2\pi i\int_D\alpha$$
for every $\alpha\in F^1H^2(X)$. A similar definition of regulator is given by
$$\reg:CH^n(X,1)\to I_2(X)_\prim:=(F^1H^2(X)_\prim)^*/H_2(X,\Z(1))_\prim.$$
using the primitive part $H^2_\prim$ of $H^2$.

\subsection{Collino cycles}\label{cc}
A Collino cycle $Z$ is a canonical higher cycle, depending on the choice of two Weierstrass points of a genus $g$ hyperelliptic curve $C$, on the Jacobian $JC$ of $C$ \cite{collino}. Fix Weierstrass points $q_1$, $q_2$, let $h$ be a degree 2 morphism
$$h:C\to\P^1$$
such that $h(q_1)=0$ and $h(q_2)=\infty$. Recall that for each $x\in C$, we denote by $C_x$ the image under the Abel-Jacobi map
\begin{align*}
    \nu_x:C &\to JC \\
    y &\mapsto y-x.
\end{align*}
Let $h_s:=h\circ\nu_{q_s}^{-1}$ be a function on $C_s:=C_{q_s}=\nu_{q_s}(C)$ for $s=1,2$. Then
$$Z:=(C_1,h_1)+(C_2,h_2)\in CH^g(JC,1)$$ 
is called a Collino cycle. Its regulator is nonzero for general $C$. In fact, it can be computed using iterated integrals.

\begin{theorem}[Thm 1.1 \cite{colombo}]\label{rz}
Let $\phi$ and $\psi$ be harmonic 1-forms on $JC$ with $\psi$ of type $(1,0)$. We use the same notation for their pullback to $C$. Then
$$\reg(Z)(\phi\wedge\psi)=2\int_{C-\gamma}\log(h)\phi\wedge\psi+2\pi i\int_\gamma (\phi\psi-\psi\phi)$$
where $\gamma:=h^{-1}([0,\infty])$.
\end{theorem}

\begin{remark}
One of the forms $\psi$ being type $(1,0)$ makes $\phi\wedge\psi$ a representative of an element in $F^1H^2(JC)$.
\end{remark}

\subsection{Colombo's construction of extension class from fundamental groups}

Colombo relates the regulator image $\reg(Z)$ to an extension class $Pe$, primitive part of an extension class $e$ constructed from the fundamental groups of the punctured curves $C-\{q_1\}$ and $C-\{q_2\}$ with the same base point $p$, another Weierstrass point of $C$. More precisely, as being done for the regulator, these extension classes are expressed in terms of iterated integrals on the Jacobian $JC$ of the curve $C$, and Colombo shows that $Pe$ is a rational multiple of $\reg(Z)$. 

We first review her construction\footnote{We have another construction for these same extensions, but since Colombo's construction is already in the literature, our construction is not necessary here.} of extensions $e$ and $Pe$ from the MHS on the fundamental groups. Fix the base point $p$ for the fundamental group $\pi_1(C-\{q\},p)$. Denote by \textcolor{black}{$L_q$} the augmentation ideal\footnote{i.e. the kernel of the augmentation map $\Z\pi_1(C-\{q\},p)\to\Z$ that sends each element of the fundamental group to 1.} of the group algebra $\Z\pi_1(C-\{q\},p)$. The powers of $L_q$ gives a natural filtration
$$\cdots\subseteq L_q^{k+1}\subseteq L_q^k\subseteq\cdots\subseteq L_q\subseteq\Z\pi_1(C-\{q\},p).$$
So we have natural extensions
\begin{equation}\label{wts}
0\to(L_q/L_q^k)^*\to(L_q/L_q^{k+1})^*\to(L_q^k/L_q^{k+1})^*\to 0.
\end{equation}
The graded pieces has pure Hodge weights as
$$(L_q^k/L_q^{k+1})^*\simeq\otimes^kH^1(C,\Z).$$
To simplify notation, we denote $H^1(C,\Z)$ by $H^1$.
If $k=2$, the above sequence (\ref{wts}) becomes
$$0\to H^1\to(L_q/L_q^3)^*\to\otimes^2H^1\to 0.$$
In particular, when $q=q_s$ ($s=1,2$) is a Weierstrass point, this extension splits and has a natural retraction
$$r_s:(L_{q_s}/L_{q_s}^3)^*\to H^1.$$
Pushing along this retraction on the sequence (\ref{wts}) for $k=3$, $q=q_s$ 
\begin{center}
\begin{tikzcd}
0\ar[r] & (L_{q_s}/L_{q_s}^3)^*\ar[r] \ar[d, "r_s"] & (L_{q_s}/L_{q_s}^4)^*\ar[r] & (L_{q_s}/L_{q_s}^3)^*\ar[r] & 0\\
0\ar[r] & H^1\ar[r] & E_s\ar[r] & \otimes^3H^1 \ar[u, dash, "\simeq"] \ar[r] & 0
\end{tikzcd}
\end{center}
we get extension class $e_s\in\Ext_\MHS(\otimes^3H^1,H^1)$ represented by $E_s$ for $s=1,2$.

We introduce several natural morphisms of MHS:
\begin{enumerate}
\item
Tensoring with the polarization $\Omega$:
$$J_\Omega: H^1(-1)\to\otimes^3H^1.$$
\item
The surjection:
$$\Pi:\otimes^2H^1\to\Z(-1),$$
which is the composition of the cup product with the isomorphism $H^2(C,\Z)\simeq\Z(-1)$.
\item
The standard inclusion:
$$\iota:\textcolor{black}{\Lambda^2H^1}\to \otimes^2H^1.$$
\item
The integration over $C$:
$$\int_C:\textcolor{black}{\Lambda^2H^1}\to\Z$$
which is the map $\Pi\circ\iota$ up to a Tate twist $\Z(-1)$. Since we have natural isomorphism $\Lambda^2H^1\simeq H^2(JC,\Z)$, we can identify the kernel
$$\ker\int_C\cong H^2(JC)_\prim$$
with the primitive part of $H^2(JC,\Z)$. By Carlson, we can identify
$$\Ext^1_\MHS(\textcolor{black}{\Lambda^2H^1},\Z)\cong I_2(JC)$$
and 
$$\Ext^1_\MHS(\ker\int_C,\Z)\cong I_2(JC)_\prim.$$
\end{enumerate} 

Now we are ready to construct the extension classes $e$ and $Pe$, represented by the extensions $E$ and $PE$ respectively in the following diagram.
\begin{center}
\begin{tikzcd}
0 \ar[r] & H^1 \ar[r] \ar[d, equal] & E_2-E_1 \ar[r] & \otimes^3H^1 \ar[r] & 0 \\
0 \ar[r] & H^1 \ar[r] \ar[d, "\otimes H^1"] & E_\Omega \ar[r] \ar[d, "\otimes H^1"] & H^1(-1) \ar[r] \ar[u, "J_\Omega"] \ar[d, "\otimes H^1"] & 0 \\
0 \ar[r] & \otimes^2H^1 \ar[r] \ar[d, "\Pi"] & E_\Omega\otimes H^1 \ar[r] & \otimes^2H^1(-1) \ar[r] \ar[d, equal] & 0 \\
0 \ar[r] & \Z(-1) \ar[r] \ar[d, "\otimes\Z(1)"] & \widetilde{E}(-1) \ar[r] \ar[d, "\otimes\Z(1)"] & \otimes^2H^1(-1) \ar[r] \ar[d, "\otimes\Z(1)"] & 0 \\
0 \ar[r] & \Z \ar[r] \ar[d, equal] & \widetilde{E} \ar[r] & \otimes^2H^1 \ar[r] & 0 \\
0 \ar[r] & \Z \ar[r] & E \ar[r] & \Lambda^2H^1 \ar[r] \ar[u, "\iota"] & 0 \\
0 \ar[r] & \Z \ar[r] \ar[u, equal] & PE \ar[r] & \ker\displaystyle\int_C \ar[r] \ar[u, hook] & 0 
\end{tikzcd}
\end{center}

\begin{theorem}[Thm 2.1 \cite{colombo}]\label{perz}
Let $C$ be a hyperelliptic curve with Weierstrass points $q_1$, $q_2$ and $p$. Let $h$ be a degree 2 morphism $h:C\to\P^1$ such that $h(q_1)=0$ and $h(q_2)=\infty$. Then we have
$$e=(2g+1)\Big(\reg(Z)+\log(h(p))\int_C\Big)\in I_2(JC)$$
and 
$$Pe=(2g+1)\reg(Z)\in I_2(JC)_\prim.$$
\end{theorem}
\begin{remark}
Although the computation of $Pe$ involves the base point $p$, it only depends on $q_1$ and $q_2$, since it is a rational multiple of $\reg(Z)$, whose construction does not involve $p$.
\end{remark}

\subsection{The normal functions and their induced monodromies}

One can extend the constructions of both $\reg(Z)$ and $Pe$ for a hyperelliptic curve to families of hyperelliptic curves \cite[\S 4, \S 5]{colombo}. They give rise to normal functions, i.e. sections on variations of mixed Hodge structures over the hyperelliptic Torelli space $\H_g[0]$. We make this precise in the next paragraph.

Recall that the hyperelliptic Torelli space $\H_g[0]$ is the moduli space of hyperelliptic curves of genus $g$ with a fixed symplectic basis of homology. Let 
$$\pi: \textcolor{black}{\cC_{\H_g[0]}\to\H_g[0]}\qquad\text{and}\qquad\I_{2\prim}\to\H_g[0]$$
be the universal hyperelliptic curve and the bundle over the hyperelliptic Torelli space $\H_g[0]$, whose fiber over the moduli point $[C;\{a_j,b_j\}_{j=1}^g]\in\H_g[0]$ is $I_2(JC)_\prim$. After choosing sections of Weierstrass points 
\begin{center}
    \begin{tikzcd}
    \cC \ar[d]\\
    \H_g[0] \ar[u, bend left=40, "\substack{q_1\\q_2}"] \ar[u, bend right=40, "p"'] 
    \end{tikzcd}
\end{center}
we can construct $\reg(Z)$ and $Pe$ for each fiber, and they assemble to form sections $r_\cZ$ and $P\cE$ of the family $\I_{2\prim}$.

\begin{center}
    \begin{tikzcd}
    \I_{2\prim} \ar[d] & \I_{2\prim} \ar[d] \\
    \H_g[0] \ar[u, bend left=40, "r_\cZ"] & \H_g[0] \ar[u, bend right=40, "P\cE"']
    \end{tikzcd}
\end{center}
Note that by Hain\cite{hain3}, $\I_{2\prim}$ is an admissible variation of mixed Hodge structures, since each of its fiber $I_2(JC)_\prim$ is naturally constructed from fundamental group. Each fiber $I_2(JC)_\prim$ can be identified with
$$I_2(JC)_\prim\cong\Ext^1(H^2(JC)_\prim,\Z)\cong\Ext^1(\Z,H^2(JC)_\prim(2))\cong J(H^2(JC)_\prim(2))$$
using Poincar\'e duality in the middle, where
$H^2(JC)_\prim(2):=H^2(JC)_\prim\otimes\Z(2).$
By definition \ref{nf}, the sections $r_\cZ$ and $P\cE$ are normal functions.

These normal functions induce monodromies which we now compute. Note that the monodromy action of the hyperelliptic Torelli group $T\Delta_g$ is trivial, so the associated bundle $\I_{2\prim}$ is a trivial bundle. Fix a base point $[C;\{a_j,b_j\}_{j=1}^g]$. Denote by $p_{I_{2\prim}}$ the projection of $\I_{2\prim}$ to its fiber $I_2(JC)_{\prim}$. The fundamental group of the fiber $I_2(JC)_\prim$ is $\Lambda^2 H/\langle\theta\rangle$ where $H$ is the first homology of $C$ generated by the fixed basis $\{a_j,b_j\}_{j=1}^g$, and $\theta=\sum_{j=1}^g a_j\wedge b_j$. Denote by $p_\cZ$ and $p_\cE$ the compositions of the projection $p_{I_{2\prim}}$ with normal functions $r_\cZ$ and $P\cE$, and they respectively induce homomorphisms of fundamental groups
$$\pi_\cZ=(p_\cZ)_*: T\Delta_g\to\textcolor{black}{\Lambda^2 H/\langle\theta\rangle}\qquad\text{and}\qquad\pi_\cE=(p_\cE)_*: T\Delta_g\to\Lambda^2 H/\langle\theta\rangle.$$ Colombo compared these monodromies on a particular element $d_i\in T\Delta_g$. The element is chosen to be the class of a Dehn twist of a simple closed curve $c_i$ on $C$ seperating $q_1$ and $q_2$, where $c_i$ is invariant under the hyperelliptic involution.

\begin{theorem}[Cor. 5.1 \cite{colombo}]
$$\pi_\cE(d_i)=(2g+1)\pi_\cZ(d_i).$$
\end{theorem}

This theorem can easily be improved, as indicated by Colombo.
\begin{proposition}
$$\pi_\cE=(2g+1)\pi_\cZ.$$
\end{proposition}
\begin{proof}
By Theorem \ref{perz}, $Pe$ and $(2g+1)\reg(Z)$ are identified on each fiber $I_2(JC)_\prim$ of $\I_{2\prim}\to\H_g[0]$, so they induce the same homomorphisms of fundamental groups
$$(P\cE)_*=(2g+1)(r_\cZ)_*.$$
The result follows from composing this with $(p_{I_{2\prim}})_*$, which is the identity map.
\end{proof}

Moreover, using Colombo's computation the monodromies on particular generators $d_i$ of the hyperelliptic Torelli group $T\Delta_g$ \cite[Cor. 4.2]{colombo}, we have the following result.
\begin{theorem}\label{dehn image normal function}
Let $d_i\in T\Delta_g$ be the class of the Dehn twist of a separating simple closed curve $c_i$ on $C$, where $c_i$ is invariant under the hyperelliptic involution, then
    $$\pi_\cZ(d_i)=\begin{cases}
        0 & \text{if $c_i$ does not separate $q_1$ and $q_2$}\\
        4\theta_i'' & \text{if $c_i$ separates $q_1$ and $q_2$}
    \end{cases}$$
\end{theorem}
\begin{proof}
     We follow the same steps as in \cite[Prop. 4.1]{colombo} to compute. We provide full details here so that interested readers can compare it with Colombo's computation.
    
    First, we set some notations. Pick a base point $[C]\in H$. Let $d\in T\Delta_g$ be the class of a Dehn twist $D_d$ of a separating simple closed curve $c$ on $C$, invariant under the hyperelliptic involution. Let $\lambda_d$ be the loop in $H$ based at $[C]$, that corresponds to the Dehn twist $D_d$. This loop lifts to a path $\Tilde{\lambda}_d:[0,1]\to\Tilde{H}$ in the universal covering $\Tilde{H}$ of $H$ with $\Tilde{\lambda}_d(0)=[C]$ and $\Tilde{\lambda}_d(1)=[D_dC]$. There is a universal family of hyperelliptic curves over the path $\Tilde{\lambda}_d$ and we denote the fiber over $\Tilde{\lambda}_d(t)$ by $C_t$. In particular, $C_0=C$. The sections $q_1$ and $q_2$ over $\lambda_d$ lift to $\Tilde{q}_1$ and $\Tilde{q}_2$ over $\Tilde{\lambda}_d$, which on each fiber $C_t$ correspond respectively to the zero and the pole of a degree 2 map $h_t:C_t\to\P^1$ that we chose to construct the Collino cycle. Let $\gamma_t:=h_t^{-1}([0,\infty])$ be the path on $C_t$, then $h_0=h$ and $\gamma_0=\gamma$ are the same as those defined in Thm. \ref{rz}. The section, i.e. normal function, $r_\cZ$ restricts to the loop $\lambda_d$ in $H$, and it can be lifted to a normal function $\Tilde{r}_\cZ$ along $\Tilde{\lambda}_d$ in $\Tilde{H}$. 
    
    Now we compute the monodromy. For any $t\in[0,1]$, $\Tilde{r}_\cZ(t)$ is the regulator of the Collino cycle constructed from Weierstrass points $\Tilde{q}_1(t)$ and $\Tilde{q}_2(t)$. By Thm. \ref{rz}, we have
    $$\Tilde{r}_\cZ(t)(\phi_t\wedge\psi_t)=2\int_{C_t-\gamma_t}\log(h_t)\phi_t\wedge\psi_t+2\pi i\int_{\gamma_t} (\phi_t\psi_t-\psi_t\phi_t)$$
    for closed 1-forms $\phi_t$ and $\psi_t$ on $C_t$, with $\psi_t$ of type $(1,0)$. By covering theory, we have 
    $$\pi_\cZ(d)=\frac{1}{2\pi}\left[\Tilde{r}_\cZ(1)-\Tilde{r}_\cZ(0)\right]\in \Lambda^2 H/\langle\theta\rangle.$$
    We can choose $\phi_1=\phi_0=:\phi$ and $\psi_1=\psi_0=:\psi$. So 
    \begin{align*}
        [\Tilde{r}_\cZ(1)-\Tilde{r}_\cZ(0)](\phi\wedge\psi)
        &=2\left(\int_{C_1-\gamma_1}\log(h_1)\phi\wedge\psi-\int_{C-\gamma}\log(h)\phi\wedge\psi\right)\\
        &\quad +2\pi i\left(\int_{\gamma_1} (\phi\psi-\psi\phi)-\int_{\gamma} (\phi\psi-\psi\phi)\right)
    \end{align*}

    \emph{Case}(i): If $c=c_i$ does not separate $q_1$ and $q_2$, then it is easy to see that we can choose the same branch for the logarithm $$\log(h_1)=\log(h)$$ because $q_1$ and $q_2$ are on the same subsurface. Moreover, $\gamma$ does not intersect with $c_i$, so that the Dehn twist does not change $\gamma$ and $$\gamma_1=\gamma.$$
    Therefore, on the right hand side of the above equation, both terms vanish and we have
    $$\pi_\cZ(d)=0.$$

    \emph{Case}(ii): If $c=c_i$ separates $q_1$ and $q_2$, then the result follows directly from \cite[Cor. 4.2]{colombo}. The essential changes are that of choosing a different branch of $\log(h)$ and that the Dehn twist carrying $\gamma$ to $\gamma_1=\gamma+2d$.
\end{proof}


\section{Hyperelliptic Johnson homomorphisms and Collino classes}
In this section, we relate the results in the previous sections, with a point of view from relative completion. 
Recall that by Proposition \ref{derp w -2}, we have the decomposition $\Der_{-2}\p = V_{[2^2]} + V_{[1^2]}$ as an $\Sp(H)$-module. Set $V = V_{[1^2]}$ and $V' = V_{[2^2]}$ for simplicity. 
Recall that $\tilde\theta$ is the projection $\Lambda^2H \to \Lambda^2H/\langle \theta\rangle \cong V$. Denote the composition $\tilde\theta\circ\pi_{\Lambda^2H}$ (see \S 4) by $\tilde\pi_{\Lambda^2H}$. 
\subsection{A Collino class in $H^1(\Delta_{g,2}, V)$}
Let $q_1$ and $q_2$ be distinct Weierstrass points in $S$. Recall from \ref{weierstrass class} that we have the classes $[q_1]$ and $[q_2]$ in $H^1(\Delta_{g,2}, V)$ given by their corresponding hyperelliptic Johnson homomorphisms $\tau^\hyp_{q_i}$. Let $\zeta = [q_2] - [q_1]$ in $H^1(\Delta_{g,2}, V)$. Via the isomorphism $H^1(\Delta_{g,2}, V) \cong \Hom_{\Sp(H)}(H_1(\u_{g,2}), V)$, the class $\zeta$ corresponds to the $\Sp(H)$-equivariant map $$\tilde\tau^\adj_\zeta:= \tilde\tau^\adj_{q_2} - \tilde\tau^\adj_{q_1} : H_1(\u_{g,2})\to V.
$$ 
Denote the map $\tilde\tau^\hyp_{q_2}-\tilde\tau^\hyp_{q_1}: T\Delta_g\to V$ by $\tilde\tau^\hyp_\zeta$.
Note that there is a commutative diagram
$$
\xymatrix{
T\Delta_g\ar[dr]^{\tilde\tau^\hyp_\zeta}\ar[d]_{r^\ab}&\\
H_1(\u_{g,2})\ar[r]_-{\tilde\tau^\adj_\zeta}&V,
}
$$
where the map $r^\ab$ is induced by the relative completion of $\Delta_{g,2}$. 

\begin{theorem}\label{diff of hyp johnson}
With notation as above, if \textcolor{black}{$g \geq 2$}, we have
$$
\tilde\tau^\hyp_\zeta = \textcolor{black}{(g+1)} \pi_\cZ
$$
\end{theorem}
\begin{proof} Suppose that $g \geq 2$. 
Let $D$ in $T\Delta_g$ be the class of the Dehn twist along a simple separating curve $d$. In the case where $d$ does not separate $q_1$ and $q_2$, it follows from Proposition \ref{dehn twist image} and \ref{dehn image normal function} that both $\tilde\tau^\hyp_\zeta(D)$ and $\pi_\cZ(D)$ are zero. 
So consider the case where $d$ separates $q_1$ and $q_2$. Say $S$ is divided by $d$ into two subsurfaces $S'_i$ of genus $i$ and $S''_i$ of genus $g-i$, which contain $q_1$ and $q_2$, respectively, as in \ref{johnson hyp nonzero}.  Fix symplectic bases $a_1, \ldots, a_{i}, b_1,\ldots, b_{i}$ and $a_{i+1}, \ldots, a_g, b_{i+1}, \ldots, b_g$ for $S'_i$ and $S''_i$, respectively. Let $\theta'_i = \sum_{\ell=1}^{i}a_\ell\wedge b_\ell$ and $\theta''_i =\sum_{\ell=i+1}^g a_\ell\wedge b_\ell$. We have $\theta =\theta'_i +\theta''_i$.  Then by Proposition \ref{dehn twist image}, the image $(\tau^\hyp_{q_2} - \tau^\hyp_{q_1})(D)$ can be represented as 
$$
\zeta_D:=\frac{1}{2}\phi((\theta'_i)^2 - (\theta''_i)^2).
$$ An easy computation using Lemma \ref{proj formula onto [1^2]} shows that the projection of the derivation $\zeta_D$ onto $V$ by $\tilde\pi_{\Lambda^2H}$ is given by 
$$\tilde\pi_{\Lambda^2H}(\zeta_D) = 2(2g+2)\theta''_i \,\,\,\,\mathrm{mod}\,\, \theta.$$
By Theorem \ref{dehn image normal function}, we have $\pi_\cZ(D) = 4\theta''_i$, and hence 
$$
\tilde\tau^\hyp_\zeta(D) = (g+1)\pi_\cZ(D).
$$
Our claim follows from Theorem \ref{hyp torelli generators} stating that $T\Delta_g$ is generated by the classes of Dehn twists along symmetric separating simple closed curves. 
\end{proof}
\begin{remark}
Consequently, the Collino cycle $(Z, q_1, q_2)$ determined by $q_1$ and $q_2$ yeilds a nontrivial class in $H^1(\Delta_{g,2}, V)$, where $\Delta_{g,2}$ fixes $q_1$ and $q_2$. 
\end{remark}

 \subsection{Weierstrass subspace of $H^1(\Delta_g[2], V)$}
 Recall from \ref{level 2 hyp map class fixing w points} that $\Delta_g[2]$ fixes all of the Weierstrass points. Therefore for each Weierstrass point $q$, we have the hyperelliptic Johnson homomorphism $\tau^\hyp_q$.  
 
\begin{lemma}\label{key lemmma} For two distinct Weierstrass points $q_1$ and $q_2$, 
the image of 
$\tau^\hyp_{q_2} - \tau^\hyp_{q_1}$ is contained in $V$.
\end{lemma}
\begin{proof}
With notation from the proof of Theorem \ref{diff of hyp johnson}, we will show that the projection of $\zeta_D$ in the $V'$-component of $\Der_{-2}\p$ is trivial. Let $\zeta'_D$ be the $V'$-part of $\zeta_D$ in $\Der_{-2}\p$.    We have
$$
(\hat\theta\circ\pi_{\Lambda^2H})(\zeta_D) = 2(2g+2)\left(\theta''_i-\frac{g-i}{g}\theta\right),
$$
where $\hat\theta: \Lambda^2H \to V$ is the projection given by $u\wedge v - \frac{\langle u, v\rangle}{g}\theta$ (see \S 4).
Then by Lemma \ref{tambo proj}, the vector 
$$
\frac{1}{2}((\theta'_i)^2 - (\theta''_i)^2) + \theta''_i\theta - \frac{g -i}{g}\theta^2
$$
maps into the  $V'$-component of $\Der_{-2}\p$ under $\phi$. By Proposition \ref{theta_I image}, the image of $\theta^2$ is trivial because $\phi(\theta^2) = -2\mathrm{adj}_\theta$, and since $\theta = \theta'_i +\theta''_i$, we have 
\begin{align*}
\zeta'_D =\phi\left(\frac{1}{2}((\theta'_i)^2 -(\theta''_i)^2) +\theta''_i\theta -\frac{g-i}{g}\theta^2\right ) &= \phi\left(\frac{1}{2}((\theta'_i)^2 -(\theta''_i)^2) +\theta''_i\theta\right)\\
& = \phi\left(\frac{1}{2}(\theta'_i)^2 +\frac{1}{2}(\theta''_i)^2 + \theta'_i\theta''_i\right)\\
& = \frac{1}{2}\phi((\theta'_i)^2 + 2\theta'_i\theta''_i +(\theta''_i)^2)\\
& = \frac{1}{2}\phi(\theta^2).
\end{align*}
This implies that $\zeta'_D$ is zero in $\Der_{-2}\p$. Therefore, $\zeta_D$ is in $V$. Since $D$ is arbitrary, it follows that the image of $\tau^\hyp_{q_2} - \tau^\hyp_{q_1}$ is contained in $V$. Note that when $D$ is the Dehn twist along a simple separating curve that does not separate $q_1$ and $q_2$, $\tau^\hyp_{q_1}(D)= \tau^\hyp_{q_2}(D)$, and so $(\tau^\hyp_{q_2}-\tau^\hyp_{q_1})(D) = 0$. 
\end{proof}

We call the image of $\zeta$ in $H^1(\Delta_g[2], V)$ via the homomorphism $H^1(\Delta_{g,2}, V)\to H^1(\Delta_g[2], V)$ as a Collino class. Since $\tilde\tau^\hyp_\zeta$ is a nontrivial homomorphism, It follows that each Collino class in $H^1(\Delta_g[2], V)$ is nontrivial.  Similarly, for each Weierstrass point $q$, we call the class $[q]$ in $H^1(\Delta_{g, 1}, V)$ and its image under the homomorphism $H^1(\Delta_{g,1}, V)\to H^1(\Delta_g[2], V)$ as a Weierstrass class. Let $X_\zeta$ be the subspace of $H^1(\Delta_g[2], V)$ spanned by all Collino classes and let $X_\omega$ be the subspace of $H^1(\Delta_g[2], V)$ spanned by all Weierstrass classes. 

\begin{theorem}With notation as above, if $g \geq 2$, then $X_\zeta = X_\omega$ and $\mathrm{dim}~X_\omega =2g+1$. 
\end{theorem}
\begin{proof}
Let $q_1, \ldots, q_{2g+2}$ be the Weierstrass points of $S$  and $W =\{q_1, \ldots, q_{2g+2}\}$. Then we have the Weierstrass classes $[q_i]$ in $H^1(\Delta_g[2], V)$. Set $p = q_{2g+2}$. For each $i = 1, \ldots, 2g+1$, define $\zeta_i$ by $\zeta_i = [q_i] - [p]$. By Theorem \ref{diff of hyp johnson}, each $\zeta_i$ is a Collino class. Suppose that 
\begin{equation}
\sum_{i=1}^{2g+1}c_i\zeta_i =0 \tag{*},
\end{equation}
where each $c_i$ is in $\Q$. Then we have $\sum_{i=1}^{2g+1}c_i\tilde\tau^\adj_{\zeta_i}=0$. 
Recall that the subgroup $\Aut_pW$ of $\Aut W$ fixing $p$ is isomorphic to $\SS_{2g+1}$ and that it acts on $H^1(\Delta_g[2], V)$ and hence on $\Hom_{\Sp(H)}(H_1(\u_g[2]), V)$. Let $t$ be a cycle of length $2g+1$ in $\Aut_pW$. Then we have
\begin{align*}
\sum_{j=1}^{2g+1}t^j\left(\sum_{i=1}^{2g+1}c_i\tilde\tau^\adj_{\zeta_i}\right)&=\sum_{j=1}^{2g+1}\sum_{i=1}^{2g+1}c_i\tilde\tau^\adj_{t^j(\zeta_i)}\\
&= \sum_{j=1}^{2g+1}\left(\sum_{i=1}^{2g+1}c_i\right)\tilde\tau^\adj_{\zeta_j} \\
&=\left(\sum_{i=1}^{2g+1}c_i\right)\sum_{j=1}^{2g+1}\tilde\tau^\adj_{\zeta_j} =0.
\end{align*}
Suppose that $\sum_{i=1}^{2g+1}c_i \not =0$. Then we have $\sum_{i=1}^{2g+1}\tilde\tau^\adj_{\zeta_i} =0$. Now for every $D $ in $T\Delta_g$, we have 
$$
\tilde\theta\circ \pi_{\Lambda^2H}\left(\sum_{i=1}^{2g+1}(\tau^\hyp_{q_i} -\tau^\hyp_{p})(D)\right) = \sum_{i=1}^{2g+1}\tilde\tau^\adj_{\zeta_i}(r^\ab(D)) = 0,
$$
where $r^\ab$ is the homomorphism $T\Delta_g \to H_1(\u_g[2])$ induced by the relative completion of $\Delta_g[2]$. By Lemma \ref{key lemmma}, it then follows that 
$$
\sum_{i=1}^{2g+1}(\tau^\hyp_{q_i} -\tau^\hyp_{p})(D)=0.
$$
Thus we have $\sum_{i=1}^{2g+1}(\tau^\hyp_{q_i} -\tau^\hyp_p) = 0$. This last equation becomes 
$$
\sum_{i=1}^{2g+1}\tau^\hyp_{q_i} = (2g+1)\tau^\hyp_p.
$$
Now, as in \ref{johnson hyp nonzero}, let $D$ in $T\Delta_g$ be the Dehn twist about a simple separating closed curve $d$ separating $q_1$ and $p$ such that $d$ separates $S$ into two subsurfaces $S'_k$ of genus $k$ and $S''_k$ of genus $g-k$, which contain $q_1$ and $p$, respectively. Fix symplectic bases $a_1, \ldots, a_{k}, b_1,\ldots, b_{k}$ and $a_{k+1}, \ldots, a_g, b_{k+1}, \ldots, b_g$ for $S'_k$ and $S''_k$, respectively. Let $\theta'_k = \sum_{\ell=1}^{k}a_\ell\wedge b_\ell$ and $\theta''_k =\sum_{\ell=k+1}^g a_\ell\wedge b_\ell$. Set $x =a_{k+1}$. Then 
$$
\sum_{i=1}^{2g+1}\tau^\hyp_{q_i}(D)(x) = l\tau^\hyp_{q_1}(D)(x) = -l[\theta''_k, x],
$$
where $l$ is the number of Weierstrass points contained in $S'_k$. Since $l \geq 1$, this is a nontrivial element. On the other hand, $(2g+1)\tau^\hyp_p(D)(x) =0$, which is a contradiction. Therefore, $\sum_{i=1}^{2g+1}c_i =0$. So  substituting $c_{2g+1} = -\sum_{i=1}^{2g}c_i$ into the equation (*), we obtain

\begin{equation}
\sum_{i=1}^{2g}c_i(\zeta_i -\zeta_{2g+1}) =0 \tag{**}.
\end{equation}
Setting $p = q_{2g+1}$ and replacing $\zeta_i$ with $[q_i] - [p]$ for each $i = 1, \ldots, 2g$, the eqation $(**)$ becomes
$$
\sum_{i=1}^{2g}c_i\zeta_i=0.
$$
Inductively, we get $c_1([q_1] - [q_2]) = 0$. The class $[q_1] -[q_2]$ is nontrivial, and so $c_1 =0$. From the inductive steps, it follows that each $c_i = 0$ for $i =1, \ldots, 2g+1$. Therefore, $\mathrm{dim}~ X_\zeta \geq 2g+1$.  \\
\indent On the other hand, we observe that the vector $\sum_{i=1}^{2g+2}[q_i]$ is fixed by the action of $\SS_{2g+2}$. Thus, by Proposition \ref{hyp cohomology}, it is in $H^1(\Delta_g, V)= H^1(\Delta_g[2], V)^{\SS_{2g+2}}$. It follows from a result of Tanaka \cite{tan} that $H^1(\Delta_g, V)=0$, and hence $\sum_{i=1}^{2g+2}[q_i]=0$ and $\mathrm{dim}~X_\omega \leq 2g+1$. Since $X_\zeta \subset  X_\omega$ and $\mathrm{dim}~X_\zeta \geq 2g+1$, our claim follows. 
\end{proof}
\begin{remark}\label{main remark}
In $H^1(\Delta_g[2], V)$, each Weierstrass class is a linear combination of Collino classes. In fact, for each $i= 1, \ldots, 2g+2$, we have 
$$
\frac{1}{2g+2}\sum_{j=1}^{2g+2}([q_i] -[q_j]) = [q_i ] - \frac{1}{2g+2}\sum_{j=1}^{2g+2}[q_j] =[q_i],
$$
and so $(2g+2)[q_i]$ is an integral combination of $2g+1$ Collino classes in $H^1(\Delta_g[2], V)$. 
\end{remark}


\end{document}